\documentclass[final,leqno,onefignum,onetabnum]{siamltex1213}

\usepackage{amsmath}
\usepackage{amsfonts}
\usepackage[sans]{dsfont}
\usepackage{soul}
\usepackage{amssymb}
\usepackage{verbatim}
\usepackage{rotating}
\usepackage{placeins}
\usepackage{lscape}
\usepackage{cite}
\usepackage{enumerate}

\newtheorem{condition}[theorem]{Condition}

\newcommand{\ignore}[1]{}

\title{High-Order Implicit Time-Marching Methods Based on Generalized Summation-By-Parts Operators\footnotemark[1]} 

\author{P.~D. Boom\footnotemark[2]\ \footnotemark[3] \and D.~W. Zingg\footnotemark[2]\ \footnotemark[4]}

\begin{document}
\maketitle

\renewcommand{\thefootnote}{\fnsymbol{footnote}}
\footnotetext[1]{Some of the material presented in this article has also appeared in: \textsc{P. D. Boom and D. W. Zingg}, \textit{Investigation of Efficient High-Order Implicit Runge-Kutta Methods Based on Generalized Summation-by-Parts Operators}, 22nd AIAA Computational Fluid Dynamics Conference, AIAA-2015-2757 (2015).
}
\footnotetext[2]{Institute for Aerospace Studies, University of Toronto, Toronto, Ontario, M3H 5T6, Canada}
\footnotetext[3]{Ph.D. Candidate}
\footnotetext[4]{Professor and Director, Tier 1 Canada Research Chair in Computational Aerodynamics, J. Armand Bombardier Foundation Chair in Aerospace Flight}
\renewcommand{\thefootnote}{\arabic{footnote}}

\slugger{sisc}{xxxx}{xx}{x}{x--x}

\begin{abstract}

This article extends the theory of classical finite-difference summation-by-parts (FD-SBP) time-marching methods to the generalized summation-by-parts (GSBP) framework. Dual-consistent GSBP time-marching methods are shown to retain: A and L-stability, as well as superconvergence of integral functionals when integrated with the quadrature associated with the discretization. This also implies that the solution approximated at the end of each time step is superconvergent. In addition GSBP time-marching methods constructed with a diagonal norm are BN-stable. This article also formalizes the connection between FD-SBP/GSBP time-marching methods and implicit Runge-Kutta methods. Through this connection, the minimum accuracy of the solution approximated at the end of a time step is extended for nonlinear problems. It is also exploited to derive conditions under which nonlinearly stable GSBP time-marching methods can be constructed. The GSBP approach to time marching can simplify the construction of high-order fully-implicit Runge-Kutta methods with a particular set of properties favourable for stiff initial value problems, such as L-stability. It can facilitate the analysis of fully discrete approximations to PDEs and is amenable to to multi-dimensional spcae-time discretizations, in which case the explicit connection to Runge-Kutta methods is often lost. A few examples of known and novel Runge-Kutta methods associated with GSBP operators are presented. The novel methods, all of which are L-stable and BN-stable, include a four-stage seventh-order fully-implicit method, a three-stage third-order diagonally-implicit method, and a fourth-order four-stage diagonally-implicit method. The relative efficiency of the schemes is investigated and compared with a few popular non-GSBP Runge-Kutta methods.

\end{abstract}



\pagestyle{myheadings}
\thispagestyle{plain}
\markboth{P.~D. BOOM AND D.~W. ZINGG}{TIME-MARCHING METHODS BASED ON GSBP OPERATORS}

\section{Introduction}\label{Sec:Intro}

%
%

Recently, it was shown that finite-difference summation-by-parts (FD-SBP) operators \cite{Kreiss1974,Strand1994} and simultaneous approximation terms (SATs) \cite{Carpenter1994,Carpenter1999,Funaro1988,Funaro1991} can be used to construct high-order fully-implicit time-marching methods \cite{Nordstrom:2013,Nordstrom:2014}. An important motivation for the classical SBP-SAT approach is the ability to construct energy estimates of discrete approximations to ordinary and partial differential equations (ODEs and PDEs) \cite{Nordstrom:2013}. Provided the continuous problem is well-posed, these estimates can be used to prove the stability of the numerical solution. By definition, classical SBP time-marching methods are L-stable \cite{Nordstrom:2014} and lead to superconvergence of integral functionals \cite{Hicken:2011b,Hicken:2014}. Furthermore, those associated with diagonal norm matrices are BN-stable and energy stable \cite{Nordstrom:2014}.

Classical SBP time-marching methods can be implemented with multiple time steps using a multiblock approach. Dual-consistency enables each time step (block) to be solved sequentially in time; however, the solution points within each time step are fully coupled \cite{Nordstrom:2014}. This is analogous to a fully-implicit Runge-Kutta method. The classical FD-SBP time-marching methods considered in \cite{Nordstrom:2013, Nordstrom:2014} are constructed from a repeating centered FD stencil with boundary closures. They are defined on a uniform distribution of solution points which includes both boundary points of the time step. As a consequence, these schemes require a significant number of solution points within each time step to achieve a prescribed order of accuracy.

An extension of the SBP property for initial boundary value problems (IBVPs) was proposed in \cite{DBZ:2014}, called generalized summation-by-parts (GSBP). GSBP operators require significantly fewer solution points than classical FD-SBP operators to achieve a prescribed order of accuracy \cite{DBZ:2014}. This is accomplished by removing the need for a connection to a centered FD stencil, and requiring only that the solution points be unique. By doing so, the GSBP framework also describes several continuous and discontinuous collocated spectral-element operators. However, GSBP operators also do not require the existence of basis functions and hence have increased flexibility relative to spectral-element approaches. For example, this flexibility can be used to simplify the construction of multidimensional operators: we only require the existence of a positive cubature, as opposed to a full set of basis functions \cite{Hicken:2015}.

The objective of this paper is to extend the theory of classical FD-SBP time-marching methods \cite{Nordstrom:2013,Nordstrom:2014} to the GSBP framework \cite{DBZ:2014}. The generalized framework was specifically designed to mimic many of the characteristics of the classical approach. Therefore, the primary challenge in extending the time-marching theory is to account for operators which do not include one or both boundary points of the time step. In this article we show that GSBP time-marching methods retain: A and L-stability, as well as superconvergence of integral functionals. Moreover, those based on diagonal norms are shown to retain BN-stability and energy stability. This paper also presents the connection between SBP/GSBP time-marching methods and implicit Runge-Kutta methods. This connection is used to extend parts of the superconvergence theory to nonlinear problems. It is also used to derive conditions under which BN-stable dense-norm GSBP time-marching methods can be constructed. While SBP/GSBP time-marching method form a subset of implicit Runge-Kutta methods, the SBP/GSBP characterization remains important. The approach facilitates the analysis of fully-discrete approximations of ODEs and PDEs using high-order schemes which are unconditionally stable by definition. This can even lead to some novel Runge-Kutta schemes, as shown in the article. Furthermore, it is amenable to multi-dimensional space-time discretizations of PDEs, in which case the explicit connection to Runge-Kutta methods is often lost.

The paper is organized as follows: Section \ref{Sec:Theory} gives a brief review of the GSBP-SAT approach presented in \cite{DBZ:2014} within the context of IVPs. The classical SBP-SAT accuracy and stability theory is extended for GSBP time-marching methods in Sections \ref{Sec:Acc} and \ref{Sec:Stab}. The connection to Runge-Kutta methods is then presented in Section \ref{Sec:Runge-Kutta_Analogy}, along with some additional theoretical results. Sample GSBP time-marching methods are given in Section \ref{Sec:Methods} and numerical examples are presented in Section \ref{Sec:Results}. A summary concludes the paper in Section \ref{Sec:Conc}.

%
\section{The GSBP-SAT Approach}\label{Sec:Theory}
%

This section presents a brief review of the GSBP-SAT framework presented in \cite{DBZ:2014} within the context of IVPs.
%
%
\subsection{Generalized Summation-By-Parts Operators}\label{Sec:GSBP}
%
%
An important objective of the SBP/GSBP approach is to prove the numerical stability of a discretization by mimicking the continuous stability analysis. Consider the nonlinear IVP:
\begin{equation}\label{Eq:NonIVP}
	\mathcal{Y}^\prime = \mathcal{F}(\mathcal{Y},t),\quad 
	\mathcal{Y}(t_0) = \mathcal{Y}_0, \quad\text{with}\quad t_0 \le t \le t_f,
\end{equation}
where $\mathcal{Y}\in\mathbb{C}$, $\mathcal{Y}^\prime=\frac{d\mathcal{Y}}{dt}$, $\mathcal{F}(\mathcal{Y},t):\lbrace\mathbb{C},\mathbb{R}\rbrace\to\mathbb{C}$, and $\mathcal{Y}_0$ is the initial data. The stability of the continuous IVP \eqref{Eq:NonIVP} can be shown through the use of the energy method in which a bound on the norm of the solution, called an energy estimate, is derived with respect to the initial condition \cite{Otto,Gustafsson1996,Gustafsson2008}. This is accomplished by taking the inner product of the solution and the IVP and relating the integrals to the initial data through the use integration-by-parts (IBP):
\begin{equation}\label{Eq:IBP}
	\left( \mathcal{U}, \frac{\partial\mathcal{V}}{\partial t} \right) + 
	\left( \frac{\partial\mathcal{U}}{\partial t}, \mathcal{V} \right) = 
	\bar{\mathcal{U}}\mathcal{V}\big|_{t_0}^{t_f},
\end{equation}
where $(\mathcal{U},\mathcal{V})$ defines the L$_2$ inner product $\int_{t_0}^{t_f} \bar{\mathcal{U}} \mathcal{V} dt$, and $\bar{\mathcal{U}}$ is the complex conjugate.

In the discrete case, a first-derivative GSBP operator is defined as follows:\\[-1.5ex]

\begin{definition}\label{Def:SBPGen}
	{\bf Generalized summation-by-parts operator \cite{DBZ:2014}:} 
	A linear operator $D=H^{-1}\Theta$  is a GSBP approximation to the first derivative of order $q\ge1$ on the distribution of solution points $\mathbf{t}=[t_1,\ldots,t_n]$ with all $t_i$ are unique, if $D$ satisfies:
	\begin{equation}\label{Eq:SBP_Deg}
		D\mathbf{t}^j = j\mathbf{t}^{j-1}, \quad j\in[0,q],
	\end{equation}
	where $\mathbf{t}^j = [t_1^j,\ldots,t_n^j]^T$, $H$ is a symmetric positive definite (SPD) matrix called the norm, and $(\Theta+\Theta^T)=\tilde{E}$ is a symmetric matrix which defines a boundary operator:
	\begin{equation}\label{Eq:E_Deg}
		(\mathbf{t}^i)^T\tilde{E}\mathbf{t}^j = t_f^{i+j} - t_0^{i+j}, \quad i,j,\in[0,r],
	\end{equation}
	with $r\ge q$.
\end{definition}

\ \\[-4.5ex]

With this definition one can easily show that a GSBP operator satisfies the following: 
\begin{equation}\label{Eq:SBP}
	\mathbf{u}^*HD\mathbf{v} + 
	\mathbf{u}^*D^TH\mathbf{v} = \mathbf{u}^*\tilde{E}\mathbf{v},
\end{equation}
where $\mathbf{u}$ is the projection of the continuous function $\mathcal{U}(t)$ onto a distribution of solution points $\mathbf{t}=[t_1,\ldots,t_n]$, and $\mathbf{u}^*$ is the conjugate transpose of $\mathbf{u}$. Each term is a discrete approximation to the corresponding term in \eqref{Eq:IBP}. Hence the GSBP operator satisfies a discrete analogue of IBP.

The existence of GSBP operators and their relationship to a quadrature rule of order $\tau$ is presented in \cite{DBZ:2014,Boom:Thesis}.  The classical SBP definition is recovered when $\tilde{E} = \diag(-1,0,\ldots,0,1)$ and therefore $\mathbf{u}^*\tilde{E}\mathbf{v}$ is strictly equal to $\bar{\mathcal{U}}\mathcal{V}\big|_{t_0}^{t_f}$. In this article we make a distinction between diagonal and nondiagonal norm matrices $H$. The latter will be referred to as dense norms following \cite{DBZ:2014}, but include all nondiagonal norms whether they are strictly dense matrices or not. For example, classical FD-SBP operators with a full, or restricted-full norm are referred to as dense-norm operators.

%
\subsection{Simultaneous Approximation Terms for GSBP Operators}\label{Sec:SAT}
%

Applying a GSBP operator to the nonlinear IVP \eqref{Eq:NonIVP} requires a means to impose the initial condition, and to couple the solution in adjacent time steps. A common approach is to use simultaneous-approximation-terms (SATs), which weakly impose these conditions via penalty terms. An important motivation for the SBP/GSBP-SAT approach is compatibility with the energy method to prove numerical stability. In order for the use of SATs to be compatible with GSBP operators and the energy method, the following additional condition is imposed:\\

\begin{condition}\label{Ass:E}
	A GSBP operator $D=H^{-1}\Theta$ must satisfy the relationship:
	\begin{equation}\label{Eq:DefE}
		\Theta + \Theta^T = \tilde{E} = \chi_{t_f}\chi_{t_f}^T-\chi_{t_0}\chi_{t_0}^T,
	\end{equation}
	%
	such that
	\begin{equation}
		\chi_{t_0}^T \mathbf{t}^j = {t_0}^{j}
		\text{ and }
		\chi_{t_f}^T \mathbf{t}^j = {t_f}^{j}
		\text{ for }
		j\in[0,r\ge q],
	\end{equation}
	%
\end{condition}

\ \\[-4.5ex]


Applying the GSBP-SAT approach to the IVP \eqref{Eq:NonIVP}: $\mathcal{Y}^\prime=\mathcal{F}(\mathcal{Y},t)$ and $\mathcal{Y}(t_0) = \mathcal{Y}_0$ for $t_0\le t\le t_f$, with two time steps (step 1: $t\in[t_0,\delta]$; and step 2: $t\in[\delta,t_f]$) yields:
\begin{multline}\label{Eq:DNonIVP}
	\left[
	\begin{array}{cc}
		D^{[1]} & \mathbf{0}\\
		\mathbf{0} & D^{[2]}
	\end{array}
	\right]\left[
	\begin{array}{c}
		\mathbf{y}_{\mathrm{d}}^{[1]}\\
		\mathbf{y}_{\mathrm{d}}^{[2]}	
	\end{array}
	\right] = 
	\left[
	\begin{array}{c}
		\mathbf{f}_{\mathrm{d}}^{[1]}\\
		\mathbf{f}_{\mathrm{d}}^{[2]}
	\end{array}
	\right]
	+ \underbrace{
	\left[
	\begin{array}{c}
		\sigma \left(H^{[1]}\right)^{-1}\chi_{t_0}^{[1]}\left(\left(\chi^{[1]}_{t_0}\right)^T \mathbf{y}_{\mathrm{d}}^{[1]} - \mathcal{Y}_0\right)\\
		\mathbf{0}
	\end{array}
	\right]
	}_\text{initial condition SAT}\\[1ex]	
	+ \underbrace{
	\left[
	\begin{array}{cc}
		\left(H^{[1]}\right)^{-1}\chi_{\delta}^{[1]} & \mathbf{0}\\
		\mathbf{0} & \left(H^{[2]}\right)^{-1}\chi_{\delta}^{[2]}
	\end{array}
	\right]
	\left[
		\begin{array}{cc}
			\sigma^{[1]}    & - \sigma^{[1]} \\
			- \sigma^{[2]}  &   \sigma^{[2]} 
		\end{array}
		\right]\left[
	\begin{array}{c}
		\left(\chi_{\delta}^{[1]}\right)^T \mathbf{y}_{\mathrm{d}}^{[1]}\\
		\left(\chi_{\delta}^{[2]}\right)^T \mathbf{y}_{\mathrm{d}}^{[2]}	
	\end{array}
	\right]}_\text{time step coupling SAT},
\end{multline}
where $\sigma$, $\sigma^{[1]}$ and $\sigma^{[2]}$ are the SAT penalty parameters, and $\mathbf{f}_{\mathrm{d}}^{[m]} = \mathcal{F}(\mathbf{y}_{\mathrm{d}}^{[m]},\mathbf{t}^{[m]})$ for time step $m$. Conservation requires that the interface SAT penalty parameters satisfy $\sigma^{[1]} = \sigma^{[2]} + 1$ \cite{Carpenter1999}. With the choice of $\sigma^{[2]}=-1$, the solution in the first time step becomes independent of the solution in the second time step, and the time steps can be solved sequentially in time. If in addition $\sigma =-1$, each time step becomes dual consistent (See Section \ref{Sec:Dual}) and we will refer to the discretization of each time step as a GSBP time-marching method.

Finally, in order to guarantee a unique solution to linear scalar ODEs, we impose an extension of Assumption 1 from \cite{Nordstrom:2013}:\\[-1.5ex]

\begin{condition}\label{Ass:Rinv}
	For a GSBP operator $D=H^{-1}\Theta$ that satisfies Condition \ref{Ass:E}, all eigenvalues of $\left(\Theta-\sigma \chi_{t_0}\chi_{t_0}^T\right)$ must have strictly positive eigenvalues for $\sigma < -\frac{1}{2}$.
\end{condition}

\ 

\begin{corollary}
	For a GSBP operator $D=H^{-1}\Theta$ that satisfies Condition \ref{Ass:E}, the matrix $\left(\Theta-\sigma \chi_{t_0}\chi_{t_0}^T\right)$ is invertible for $\sigma < -\frac{1}{2}$.
\end{corollary}

\ \\[-4.5ex]

A proof of this condition for  classical second-order FD-SBP operators is presented in \cite{Nordstrom:2013}, along with numerical demonstration for higher-order diagonal-norm FD-SBP operators. Using a multiblock approach in which all operators are identical and of fixed size, this can easily be verified numerically. This property has been verified for all GSBP operators considered in \cite{DBZ:2014}, as well as the those considered in this article.

%
\section{Accuracy}\label{Sec:Acc}
%

In many numerical simulations, integral functionals of the solution, $\mathcal{J}(\mathcal{Y}) = (\mathcal{K},\mathcal{Y}) + {\alpha}\mathcal{Y}\big|_{t_f}$, are of more interest than the solution itself. Hicken and Zingg \cite{Hicken:2011b,Hicken:2014} showed that integral functionals of the solution to dual-consistent classical FD-SBP-SAT discretizations are superconvergent. Specifically, functionals integrated with the quadrature associated with the SBP operator converge with the order of the quadrature $\tau$, rather than the lower order of the operator $q$. For example, using a dual-consistent diagonal-norm FD-SBP-SAT discretization, the rate of superconvergence is twice the order of the underlying discretization. This was shown for linear IBVPs and assumes that the continuous solution is sufficiently smooth.

Lundquist and Nordstr\"om \cite{Nordstrom:2014} observed that this property extends to IVPs and implies that the numerical solution at the end of each time step $y_{\mathrm{d},n}$ is also superconvergent. Therefore, we can recover high-order approximations of the solution using a lower-order discretization. This is analogous to Runge-Kutta methods, which make use of low-order stage approximations to construct a higher-order solution update. In addition, it is the order of this update which is referred to when discussing the order of traditional time-marching methods. This further highlights the importance of the superconvergence theory for IVPs.

The primary focus of this section is to extend the classical SBP convergence theory of integral functionals to the GSBP framework. The specific challenges are to account for the fact that  $\chi_{t_0}^T\mathbf{y}$ is not in general equal to $\mathcal{Y}_0$ and to define the accuracy of a GSBP norm matrix. To complete the extension of the theory, a number of intermediate results are presented first.

%
\subsection{The Primal Problem}\label{Sec:PAcc}
%

The classical SBP-SAT convergence theory for integral functionals is defined for IVPs which are linear with respect to the solution. The truncation error of the discrete IVP plays an important role in bounding the rate of superconvergence. In this section, the GSBP-SAT discretization of the relevant linear IVP is defined along with its truncation error.

To begin, consider the scalar IVP which is linear with respect to the solution:
\begin{equation}\label{Eq:Primal}
	\mathcal{Y}^\prime = \lambda\mathcal{Y} + \mathcal{G}(t),\quad \mathcal{Y}(t_0) = \mathcal{Y}_0, \quad\text{with}\quad t_0 \le t \le t_f,
\end{equation}
where $\mathcal{Y}(t)\in\mathbb{C}$ and $\mathcal{G}(t):\mathbb{R}\to\mathbb{C}$ is a nonlinear forcing function. This will be referred to as the primal problem in subsequent sections when discussing dual-consistency. The GSBP-SAT discretization of \eqref{Eq:Primal} is:
\begin{equation}\label{Eq:DPrimal}
		D\mathbf{y}_\mathrm{d} = \lambda\mathbf{y}_\mathrm{d} + \mathbf{g} +
	\sigma H^{-1}\chi_{t_0}\left(\chi_{t_0}^T \mathbf{y}_\mathrm{d} - \mathcal{Y}_0\right).
\end{equation}
The truncation error is obtained by substituting the projection of the continuous solution $\mathbf{y}$ into \eqref{Eq:DPrimal}:
\begin{equation}
	T_e = D\mathbf{y} - \lambda\mathbf{y} - \mathbf{g} - \sigma H^{-1}\chi_{t_0}(\chi_{t_0}^T\mathbf{y} - \mathcal{Y}_0).
\end{equation}
This can be rearranged in terms of the difference between the projection of the continuous solution and the solution to the discrete equations:
\begin{equation}\label{Eq:err2}
	T_e = H^{-1}(\Theta - \sigma \chi_{t_0}\chi_{t_0}^T - \lambda H)(\mathbf{y} - \mathbf{y}_\mathrm{d}).
\end{equation}
Alternatively, it can be written with respect to the projection of the continuous solution and time derivative:
\begin{equation}\label{Eq:err}
	T_e = D\mathbf{y} - \mathbf{y}^\prime - \sigma H^{-1}\chi_{t_0}(\chi_{t_0}^T\mathbf{y} - \mathcal{Y}_0).	
\end{equation}
In the classical SBP-SAT approach the truncation error $T_e = D\mathbf{y} - \mathbf{y}^\prime$ is of order $q$ \cite{Nordstrom:2014}, as $\chi_{t_0}^T\mathbf{y}=\mathcal{Y}_0$. In the generalized case, Condition \ref{Ass:E} ensures that the projection operator $\chi_{t_0}$ is of order greater than or equal to the GSBP operator itself, $r\ge q$. Thus, the lowest-order entry in $T_e$ is also of order $q$ in the generalized case. 

%
\subsection{The Dual Problem}\label{Sec:Dual}
%

Our goal is to derive the convergence rate of integral functionals:
\begin{equation}\label{Eq:PrimalFunc}
	\mathcal{J}(\mathcal{Y}) = (\mathcal{K},\mathcal{Y}) + {\alpha}\mathcal{Y}\big|_{t_f},
\end{equation}
of the primal problem \eqref{Eq:Primal}. A key tool required to accomplish this task is the Lagrangian dual problem. The derivation of the continuous dual problem
\begin{equation}\label{Eq:Dual}
	-\Phi^\prime = \bar{\lambda}\Phi + \mathcal{K}(t),\quad \Phi({t_f}) = \bar{\alpha}, \quad\text{with}\quad t_0 \le t \le t_f.
\end{equation}
and dual functional
\begin{equation}\label{Eq:DualFunc}
	\mathcal{J}(\Phi) = \mathcal{J}(\mathcal{Y}) = (\Phi,\mathcal{G}) + \bar{\Phi}\mathcal{Y}\big|_{t_0}.
\end{equation}
can be found in several references (\textit{e.g.}\ \cite{Hicken:2011b}). This section presents a brief derivation of the discrete GSBP-SAT dual problem, along with the dual truncation error. 

To begin, consider the discrete form of the primal functional \eqref{Eq:PrimalFunc}: 
\begin{equation}\label{Eq:DFunc}
	J_H(\mathbf{y}_\mathrm{d}) = (\mathbf{k},\mathbf{y}_\mathrm{d})_H + \alpha \chi_{t_f}^T\mathbf{y}_\mathrm{d}.
\end{equation}
Subtracting the inner product of a vector $\phi_\mathrm{d}$ and the discrete primal problem \eqref{Eq:DPrimal} gives the discrete primal Lagrangian:
\begin{equation}
	\begin{array}{rl}
		L_p =& (\mathbf{k},\mathbf{y}_\mathrm{d})_H + \alpha \chi_{t_f}^T\mathbf{y}_\mathrm{d} -\\[1ex]
		
		& \quad (\phi_\mathrm{d},D\mathbf{y}_\mathrm{d} - \lambda\mathbf{y}_\mathrm{d} - \mathbf{g} - H^{-1}\sigma \chi_{t_0}(\chi_{t_0}^T\mathbf{y}_\mathrm{d} - \mathcal{Y}_0))_H.
	\end{array}
\end{equation}
Rearranging, making use of Condition \ref{Ass:E}: $\tilde{E} = \Theta + \Theta^T = \chi_{t_f}\chi_{t_f}^T - \chi_{t_0}\chi_{t_0}^T$, and simplifying yields:
\begin{equation}\label{Eq:DualInter2}
	\begin{array}{rl}
		L_p &= (\phi_\mathrm{d},\mathbf{g})_H  - \sigma\phi_\mathrm{d}^* \chi_{t_0}\mathcal{Y}_0\ +\\[1ex]
		& \hspace{0cm} \big(D\phi_\mathrm{d} + \bar{\lambda}\phi_\mathrm{d} +\mathbf{k} 
		-H^{-1}\chi_{t_f}(\chi_{t_f}^T\phi_\mathrm{d}-\bar{\alpha}) + (\sigma+1)H^{-1} \chi_{t_0}\chi_{t_0}^T\phi_\mathrm{d},\mathbf{y}_\mathrm{d}\big)_H.
	\end{array}
\end{equation}
The first two terms are an approximation of the dual functional \eqref{Eq:DualFunc}. The final term is an approximation of the inner product of the dual problem \eqref{Eq:Dual} and primal solution. 

A discretization which produces consistent approximations of both the primal and dual problems is called dual-consistent \cite{Lu2005}. As with the classical SBP-SAT approach, this occurs in the generalized case when $\sigma=-1$. For reference, the dual-consistent GSBP-SAT approximations of the dual problem and dual functional are:
\begin{equation}\label{Eq:DDual2}
	-D\phi_\mathrm{d} = \bar{\lambda}\phi_\mathrm{d} + \mathbf{k} - H^{-1}\chi_{t_f}(\chi_{t_f}^T\phi_\mathrm{d}-\bar{\alpha}),
\end{equation}
and
\begin{equation}
	J_H(\phi_\mathrm{d}) = J_H(\mathbf{y}_\mathrm{d}) = (\phi_\mathrm{d},\mathbf{g})_H  + \phi_\mathrm{d}^* \chi_{t_0}\mathcal{Y}_0.
\end{equation}

Using the same arguments made for the primal problem in Section \ref{Sec:PAcc}, the truncation error of the dual-consistent dual problem:
\begin{equation}
	\begin{array}{rl}
		\tilde{T}_e 
		&= -D\phi - \bar{\lambda}\phi - \mathbf{k} + H^{-1}\chi_{t_f}(\chi_{t_f}^T\phi - \bar{\alpha}) \\[1ex]
		&= -D\phi + \phi^\prime + H^{-1}\chi_{t_f}(\chi_{t_f}^T\phi - \bar{\alpha}),
	\end{array}
\end{equation}
is of order $q$. 

%
\subsection{The Accuracy of the Norm}\label{Sec:Norm}
%

The final piece required to prove the superconvergence of integral functionals is the accuracy with which a norm $H$ associated with GSBP operator $D=H^{-1}\Theta$ approximates the continuous L$_2$-inner product. This is characterized by the following definition: \\

\begin{definition}\label{Def:Norm}
	{\bf Accuracy of a GSBP norm:} 
	A norm $H$ associated with a GSBP operator $D=H^{-1}\Theta$ is of order $\rho$, if:
	\begin{equation}
		(\mathbf{u}, \mathbf{v})_H = \mathbf{u}^T H \mathbf{v} = \int_{t_0}^{t_f} \mathcal{U}\mathcal{V} dt, \quad \mathcal{U}\mathcal{V}\in \mathbb{P}^{\rho},
	\end{equation}
	and
	\begin{equation}
		\phi^T H \big(D\mathbf{y} + H^{-1}\chi_{t_0}(\chi_{t_0}^T\mathbf{y}-\mathcal{Y}_0)\big)
		 = \int_{t_0}^{t_f} \Phi\mathcal{Y}^\prime dt, \quad \Phi\mathcal{Y}^\prime\in \mathbb{P}^{\rho},
	\end{equation}
	where $\mathbb{P}^i$ is the polynomial space of degree $i$, $\mathcal{Y}$ is the solution to the continuous primal problem \eqref{Eq:Primal} with initial condition $\mathcal{Y}_0$, and $\Phi$ is the solution to the dual problem \eqref{Eq:Dual} with homogeneous initial condition $\alpha=0$.
\end{definition}

\ \\
For diagonal-norm GSBP operators we have $\rho=\min(2q+1,\tau)$, where $q$ is the order of the operator and $\tau\ge 2q$ is the order of the associated quadrature rule \cite{DBZ:2014}. For dense-norms GSBP operators we have $\rho=\min(2q+1,s)$, where $2\lceil q/2 \rceil \le s \le \tau$ \cite{DBZ:2014,Boom:Thesis}. A more precise definition of $s$ is given in \cite{Boom:Thesis}. Observe that $\rho$ is always greater than or equal to the order $q$ of the GSBP operator itself. However, in contrast to the classical approach, $\rho$ is not necessarily equal to the order $\tau$ of the associated quadrature.

%
\subsection{Superconvergence}\label{Sec:Super}
%

We now extend the classical SBP-SAT theory for the superconvergence of integral functionals \cite{Hicken:2011b,Hicken:2014} to the GSBP-SAT approach. Classically, the theory implies that functionals constructed from the numerical solution, as well as the solution approximated at the end of each time step, will converge with the order of the associated quadrature rule $\tau$, rather than the order of the underlying discretization $q$. For diagonal-norm FD-SBP time-marching methods these differ by a factor of two, $\tau=2q$. In this section, it is shown that GSBP time-marching methods retain the superconvergence property. In the generalized case it is the order of the GSBP norm matrix, rather than the associated quadrature rule, which dictates the rate of superconvergence. The extension of the theory is presented in a sequence of three theorems to simplify the proofs and to highlight some important intermediate results.

To begin, we consider the functional $(\mathbf{k},\mathbf{y}_\mathrm{d})_H\approx (\mathcal{K}(t),\mathcal{Y})$ associated with a dual problem that has a homogeneous initial condition, $\alpha=0$. This restriction coincides with the requirement imposed in Definition \ref{Def:Norm} for the order of a GSBP norm matrix $\rho$. The following theorem is an extension of the work presented in \cite{Hicken:2011b,Hicken:2014} for dual-consistent classical FD-SBP-SAT discretizations of IBVPs. The use of Definition \ref{Def:Norm} is the primary tool required to extend the proof. Therefore, the theorem is presented without proof:\\


\begin{theorem}\label{Thm:LinFunc1}
	If $\mathbf{y}_\mathrm{d}$ is the numerical solution of the primal problem (\ref{Eq:Primal}) with $\mathrm{Re}(\lambda)\le 0$ computed using a GSBP time-marching method of order $q$ and associated with a norm of order $\rho$, then the discrete functional $J_H(\mathbf{y}_\mathrm{d}) = (\mathbf{k},\mathbf{y}_\mathrm{d})_H$ approximates $\mathcal{J}(\mathcal{Y}) = (\mathcal{K}(t),\mathcal{Y})$ for $\mathcal{K}\mathcal{Y}\in C^{\rho}$ with order $\rho$.
\end{theorem}

\ \\
For SBP and GSBP operators, the order of diagonal norm matrices $\rho$ must be at least twice the order of the operator itself $q$ \cite{DBZ:2014}. Therefore, integral functionals will also converge with at least twice the order of the underlying time-marching method. In contrast, dense norm matrices are not required to be significantly more accurate than the GSBP operator; however, the operators themselves can often be more accurate for a fixed number of solution points. Therefore, given a distribution of solution points, the convergence of integral functionals is often similar. While the underlying solution computed using a dense-norm GSBP time-marching method can be more accurate, it will in general forfeit nonlinear stability (See Section \ref{Sec:Stab}).

To extend these results to the general integral functional $(\mathbf{k},\mathbf{y}_\mathrm{d})_H + \alpha \chi_{t_f}^T\mathbf{y}_\mathrm{d}\approx (\mathcal{K}(t),\mathcal{Y}) + \alpha\mathcal{Y}_f$, we must first prove that the solution approximated at the end of each time step $\tilde{y}_{t_f} = \chi_{t_f}^T\mathbf{y}_\mathrm{d}$ is superconvergent. Independent of the functionals, the result itself is significant as it implies that we can recover a high-order approximation of the solution from a lower-order discretization. The classical theory presented in \cite{Nordstrom:2014} does not extend naturally to the generalized case. Here, we present an alternate proof, which makes use of the following lemma proven in \cite{Boom:Thesis}:\\

\begin{lemma}\label{Lem:B}
	 A GSBP operator which satisfies Condition \ref{Ass:E}, also satisfies the identity: $\chi_{t_f}^T(\Theta+ \chi_{t_0}\chi_{t_0}^T)^{-1} = \mathds{1}^T$.
\end{lemma}

\ \\
The proof of superconvergence for the solution approximated at the end of each time step $\chi_{t_f}^T\mathbf{y}_\mathrm{d}$ is presented in the following theorem:\\[-1.5ex]

\begin{theorem}\label{Thm:SupSol}
	If $\mathbf{y}_\mathrm{d}$ is the numerical solution of the primal problem \eqref{Eq:Primal} with $\mathrm{Re}(\lambda)\le 0$ computed using a GSBP time-marching method of order $q$ and associated with a norm of order $\rho$, then the solution projected to the end of the time step $\tilde{y}_{t_f} = \chi_{t_f}^T\mathbf{y}_\mathrm{d}$ approximates $\mathcal{Y}(t_f)$ with order $\rho$.
\end{theorem}

\
\begin{proof}
	Consider a GSBP time-marching method applied to the primal problem (\ref{Eq:Primal}):
	\begin{equation}
		D\mathbf{y}_\mathrm{d} = \lambda\mathbf{y}_\mathrm{d} + \mathbf{g} -
			H^{-1}\chi_{t_0}\left(\chi_{t_0}^T \mathbf{y}_\mathrm{d} - \mathcal{Y}_0\right).
	\end{equation}
	Rearranging, and left-multiplying by $\chi_{t_f}^T$ gives
	\begin{equation}
		\tilde{y}_{t_f} = \chi_{t_f}^T\mathbf{y}_\mathrm{d} = \chi_{t_f}^T(\Theta+\chi_{t_0}\chi_{t_0}^T)^{-1}H(\lambda\mathbf{y}_\mathrm{d} + G + H^{-1}\chi_{t_0}\mathcal{Y}_0).
	\end{equation}
	Simplifying using Lemma \ref{Lem:B}: $\chi_{t_f}^T(\Theta-\sigma \chi_{t_0}\chi_{t_0}^T)^{-1} = \mathds{1}^T$, gives:
	\begin{equation}
		\tilde{y}_{t_f} = (\mathds{1},\lambda\mathbf{y}_\mathrm{d})_H + (\mathds{1},G)_H + \mathds{1}^T\chi_{t_0}\mathcal{Y}_0.
	\end{equation}
	Simplifying again using Condition \ref{Ass:E}: $\chi_{t_0}^T\mathbf{t}^i = t_0^i$ for $i\in[0,r\ge q\ge 1]$, Theorem \ref{Thm:LinFunc1}: $(\mathcal{K},\mathcal{Y})= (\mathbf{k},\mathbf{y}_\mathrm{d})_H+\mathcal{O}(\Delta t_n^{\rho})$, and Definition \ref{Def:Norm},  yields:
	\begin{equation}
		\begin{array}{rl}
			\tilde{y}_{t_f} =& (1,\lambda\mathcal{Y}) + (1,\mathcal{G}) + \mathcal{Y}_0 + \mathcal{O}(\Delta t_n^{\rho}) \\[1ex]
			
			=& \int_{t_0}^{t_f} (\lambda \mathcal{Y} +\mathcal{G})  dt + \mathcal{Y}_0 + \mathcal{O}(\Delta t_n^{\rho})	.
			
		\end{array}
	\end{equation}
	Substituting using the continuous primal problem \eqref{Eq:Primal}: $\mathcal{Y}^\prime = \lambda \mathcal{Y} +\mathcal{G}$, gives
	\begin{equation}
		\begin{array}{rl}
			\tilde{y}_{t_f} =& \int_{t_0}^{t_f} \mathcal{Y}^\prime dt + \mathcal{Y}_0 + \mathcal{O}(\Delta t_n^{\rho}) \\[1ex]
			
			=& \mathcal{Y}(t_f) + \mathcal{O}(\Delta t_n^{\rho}),
		\end{array}
	\end{equation}
	thus completing the proof.
\end{proof}

\ \\
Thus, we can construct high-order approximations of the solution at the end of each time step $\tilde{y}_{t_f} = \chi_{t_f}^T\mathbf{y}_\mathrm{d}$ from a lower-order discretization. With this result, the superconvergence of the general integral functional $\mathcal{J}(\mathcal{Y}) = (\mathcal{K}(t),\mathcal{Y}) + \alpha\mathcal{Y}_f$ follows immediately by combining Theorems \ref{Thm:LinFunc1} and \ref{Thm:SupSol}:\\[-1.5ex]

\begin{theorem}\label{Thm:LinFunc2}
	If $\mathbf{y}_\mathrm{d}$ is the numerical solution of the primal problem (\ref{Eq:Primal}) with $\mathrm{Re}(\lambda)\le 0$ computed using a GSBP time-marching method of order $q$ and associated with a norm of order $\rho$, then the discrete functional $J_H(\mathbf{y}_\mathrm{d}) = (\mathbf{k},\mathbf{y}_\mathrm{d})_H + \alpha \chi_{t_f}^T\mathbf{y}_\mathrm{d}$ approximates $\mathcal{J}(\mathcal{Y}) = (\mathcal{K}(t),\mathcal{Y})+\alpha\mathcal{Y}(T)$ for $\mathcal{K}(t)\in C^\rho$ with order $\rho$.
\end{theorem}

\ \\
In summary, we have shown that both integral functionals of the solution and approximations of the solution itself at the end of each time step are superconvergent, provided the solution is sufficiently smooth and the discretization is dual-consistent. The rate of superconvergence is related to the accuracy with which the norm matrix can approximate the continuous L$_2$ inner product, in particular inner products of the primal and dual problems. For diagonal-norm GSBP time-marching methods the rate of superconvergence is at least twice the order of the operator itself. For dense-norm GSBP time-marching methods, superconvergence is often less notable; however, the methods can be of higher-order to begin with \cite{DBZ:2014,Boom:Thesis,Strand1994}. One potential drawback of dense-norm discretizations is the loss of nonlinear stability (See Section \ref{Sec:Stab}).

Relative to classical FD-SBP time-marching methods, GSBP schemes can be constructed of higher order for a given number of solution points \cite{DBZ:2014}. This includes both the order of the operator itself as well as the norm matrix. Therefore, the extended superconvergence theory implies that GSBP time-marching methods not only have the potential to be more efficient for computing the pointwise solution, but the computation of integral functionals as well.

%
\section{Stability}\label{Sec:Stab}
%

Thus far it has been shown that GBSP time-marching methods retain the superconvergence of the classical SBP approach. However, these methods are only desirable for stiff IVPs if they also maintain the same stability properties. In this section, several linear and nonlinear stability criteria are considered. The analysis of these criteria for classical FD-SBP time-marching methods was presented in \cite{Nordstrom:2014}. The primary challenge to extend the theory for GSBP time-marching methods is to account for the fact that $\chi_{t_0}^T\mathbf{y}$ is not in general equal to $\mathcal{Y}_0$.

%
\subsection{Linear Stability}\label{Sec:LStab}
%

While many problems of interest are nonlinear, a locally linear assumption is often a reasonable approximation. An example is the solution to the compressible Navier-Stokes equations. As a result, linear stability is of significant interest. To begin the discussion of linear stability, consider the scalar linear IVP,
\begin{equation}\label{Eq:LIVP}
	\mathcal{Y}^\prime = \lambda\mathcal{Y},\quad \mathcal{Y}(t_0) = \mathcal{Y}_0, \quad\text{with}\quad t_0 \le t \le t_f,
\end{equation}
where $\mathcal{Y}\in\mathbb{C}$, and $\lambda$ is a complex constant. It is well known that \eqref{Eq:LIVP} is inherently stable for $\mathrm{Re}(\lambda)\le 0$ (\textit{e.g.}\ \cite{Lomax:2001,Dahlquist:1963,Nordstrom:2014}). A numerical method applied to (\ref{Eq:LIVP}) is called A-stable if $\mathrm{Re}(\lambda)\le 0$ implies that 
\begin{equation}
	| \tilde{y}_{t_f} | \le | \mathcal{Y}_0 |,
\end{equation}
where $\tilde{y}_{t_f} \approx \mathcal{Y}(t_f)$ is an approximation of the solution at $t_f$. Classical SBP time-marching methods were shown to be A-stable in \cite{Nordstrom:2014}, where $\tilde{y}_{t_f}$ is obtained from $\mathbf{y}_{\mathrm{d},n}$. The proof extends naturally taking $\tilde{y}_{t_f} = \chi_{t_f}^T\mathbf{y}_{\mathrm{d}}$. Thus, we present the following theorem without proof:\\[-1.5ex]

\begin{theorem}\label{Thm:AStb}
All GSBP time-marching methods are A-stable.
\end{theorem}

\ \\
A-stability guarantees that all modes are stable; however those associated with large eigenvalues may be damped very slowly. Hence, a stronger condition known as L-stability \cite{Ehle:1969} is often preferred. A numerical method applied to (\ref{Eq:LIVP}) is called L-stable if it is A-stable, and furthermore $\mathrm{Re}(\lambda)\le 0$ implies that 
\begin{equation}
	| \tilde{y}_{t_f} |\to 0 \text{ as } |\lambda|  \to \infty.
\end{equation}	
Classical SBP time-marching methods were also shown in \cite{Nordstrom:2014} to have this property; however, the proof does not extended to the generalized case. Hence, we present an alternate approach inspired by Proposition 3.8 in \cite{HairerII}. The proof is simplified by first introducing the following lemma proven in \cite{Boom:Thesis}: \\

\begin{lemma}\label{Lem:U}
	 A GSBP operator which satisfies Condition \ref{Ass:Rinv}, also satisfies the identity: $(\Theta + \chi_{t_0}\chi_{t_0}^T)^{-1}\chi_{t_0} = \mathds{1}$.
\end{lemma}

\ \\
The L-stability of GSBP time-marching methods is now proven:\\[-1.5ex]

\begin{theorem}\label{Thm:LStb}
All GSBP time-marching methods which satisfy Condition \ref{Ass:Rinv} are L-stable.
\end{theorem}

\
\begin{proof}
	A-stability follows from Theorem \ref{Thm:AStb}. What remains is to prove is $| \tilde{y}_{t_f} |\to 0 \text{ as } |\lambda| \to \infty$. Applying a GSBP time-marching method to the linear IVP \eqref{Eq:LIVP} gives:
	\begin{equation}\label{Eq:DLIVP}
		D\mathbf{y}_\mathrm{d} = \lambda\mathbf{y}_\mathrm{d} - 
		H^{-1}\chi_{t_0}\left(\chi_{t_0}^T \mathbf{y}_\mathrm{d} - \mathcal{Y}_0\right).
	\end{equation}
	Rearranging using Condition \ref{Ass:Rinv}: $(\Theta+\chi_{t_0}\chi_{t_0}^T)$ has strictly positive eigenvalues and is therefore invertible, and Lemma \ref{Lem:U}: $(\Theta + \chi_{t_0}\chi_{t_0}^T)^{-1}\chi_{t_0} = \mathds{1}$, yields:
	\begin{equation}\label{Eq:Lstb1}
		\mathbf{y}_\mathrm{d} = [I-\lambda(\Theta+\chi_{t_0}\chi_{t_0}^T)^{-1}H]^{-1}\mathds{1}\mathcal{Y}_0.
	\end{equation}
	In a similar fashion, an expression for $\tilde{y}_{t_f}$ can be constructed:
	\begin{equation}\label{Eq:Lstb2}
			\tilde{y}_{t_f} = \chi_{t_f}^T\mathbf{y}_\mathrm{d} = \chi_{t_f}^T(\Theta+\chi_{t_0}\chi_{t_0}^T)^{-1}H\lambda\mathbf{y}_\mathrm{d} + 
		\chi_{t_f}^T\mathds{1}\mathcal{Y}_0.
	\end{equation}
	Inserting (\ref{Eq:Lstb1}) into (\ref{Eq:Lstb2}) and applying Condition \ref{Ass:E}: $\chi_{t_f}^T\mathbf{t}^i=t_f^i$ for $i\in[0,r\ge q\ge 1]$, yields:
	\begin{equation}
			\tilde{y}_{t_f} = (1 + \lambda \chi_{t_f}^T(\Theta+\chi_{t_0}\chi_{t_0}^T)^{-1}H [I-\lambda (\Theta+\chi_{t_0}\chi_{t_0}^T)^{-1}H]^{-1}\mathds{1})\mathcal{Y}_0.
	\end{equation}
	Taking the limit as $|\lambda|\to\infty$:
	\begin{equation}
			\tilde{y}_{t_f} = (1 - \chi_{t_f}^T(\Theta+\chi_{t_0}\chi_{t_0}^T)^{-1}H [(\Theta+\chi_{t_0}\chi_{t_0}^T)^{-1}H]^{-1}\mathds{1})\mathcal{Y}_0.
	\end{equation}
	and simplifying using Condition \ref{Ass:E}: $\chi_{t_f}^T\mathbf{t}^i=t_f^i$ for $i\in[0,r\ge q\ge 1]$, yields:
	\begin{equation}
			\tilde{y}_{t_f} = (1 - \chi_{t_f}^T\mathds{1})\mathcal{Y}_0 = 0,
	\end{equation}
	completing the proof.
\end{proof}

\ \\[-4.5ex]

In summary, all GSBP time-marching methods are unconditionally stable for linear problems, and furthermore provide damping of stiff parasitic modes. These conditions are derived for linear problems, but are often sufficient for nonlinear problems as well. 

%
\subsection{Nonlinear Stability and Contractivity}\label{Sec:NLStab}
%

While linear stability is often sufficient, there are cases in which nonlinear stability is required. In this section, we show that GSBP time-marching methods associated with a diagonal norm matrix retain BN-stability and energy stability.

To begin, consider a subset of the general IVP (\ref{Eq:NonIVP}) which satisfy the one-sided Lipschitz condition \cite{Dekker:1984}:
\begin{equation}\label{Eq:OSLipschitz}
	\mathrm{Re}[(\mathcal{F}(\mathcal{Y},t) - \mathcal{F}(\mathcal{Z},t)\ ,\ \mathcal{Y}-\mathcal{Z})_P] \le 
	\nu || \mathcal{Y}-\mathcal{Z} ||^2,\quad \forall\mathcal{Y},\mathcal{Z}\in\mathbb{C}^M,\text{ and }t\in\mathbb{R}
\end{equation}
where $\nu\in\mathbb{R}$ is the one-sided Lipschitz constant, and $P$ is an SPD matrix defining a discrete inner product and norm over $\mathbb{C}^M$:
\begin{equation}\label{Eq:DNorm}
	(\mathcal{Y},\mathcal{Z})_P = \mathcal{Y}^*P\mathcal{Z}, \quad
	||\mathcal{Y}||^2_P = \mathcal{Y}^*P\mathcal{Y}.
\end{equation}
An IVP (\ref{Eq:NonIVP}) is said to be contractive if it satisfies the one-sided Lipschitz condition with $\nu \le 0$. The significance of this condition is that the distance between any two solutions, $|| \mathcal{Y}(t)-\mathcal{Z}(t) ||$, does not increase with time \cite{HairerII}. A numerical method applied to the IVP (\ref{Eq:NonIVP}) is called BN-stable\footnote{BN-stability is sometimes referred to as B-stability when the distinction between autonomous and non-autonomous ODEs is not made (Compare Definitions 2.9.2 and 2.9.3 of \cite{Jackiewicz:Book} and Definition 12.2 in\cite{HairerII}).} if the one-sided Lipschitz condition (\ref{Eq:OSLipschitz}) with $\nu\le 0$ implies that 
\begin{equation}
	|| \tilde{y}_{t_f}-\tilde{z}_{t_f} ||_P \le || \mathcal{Y}_0-\mathcal{Z}_0 ||_P,
\end{equation}
where $\tilde{y}_{t_f}$ and $\tilde{z}_{t_f}$ are approximations of the solutions at time $t_f$ given initial data $\mathcal{Y}_0$ and $\mathcal{Z}_0$ respectively.

A similar nonlinear stability definition was presented in \cite{Burrage:1980} for autonomous IVPs with monotonic functions. An extension of this idea for non-autonomous IVPs was introduced in \cite{Nordstrom:2014}, called energy stability. A numerical method is called energy stable if
\begin{equation}\label{Eq:NStab}
	\mathrm{Re}[(\mathcal{Y}(t),\mathcal{F}(\mathcal{Y},t))_P] \le 0,\quad 
	\forall\mathcal{Y}\in\mathbb{C}^M,\text{ and }t\in\mathbb{R}
\end{equation}
implies that 
\begin{equation}
	|| \tilde{y}_{t_f} ||_P \le || \mathcal{Y}_0 ||_P.
\end{equation}

Classical diagonal-norm FD-SBP time-marching methods have been shown to be both BN-stable and energy stable, where $\tilde{y}_{t_f} = y_{\mathrm{d},n}$ \cite{Nordstrom:2014}. The proofs extend immediately for GSBP time-marching methods taking $\tilde{y}_{t_f} = \chi^T_{t_f}\mathbf{y}_\mathrm{d}$. Therefore, the following Theorem is presented without proof:\\[-1.5ex]

\begin{theorem}\label{Thm:EngyStb_diag}
All diagonal-norm GSBP time-marching methods are BN-stable, energy stable, and hence monotonic.
\end{theorem}

\ \\[-4.5ex]

As with the classical SBP time-marching methods, nonlinear stability does not extend in general to dense-norm GSBP time-marching methods. However, in Section \ref{Sec:Runge-Kutta_Analogy} conditions are derived under which nonlinearly stable dense-norm GBP time-marching methods can be constructed.

%
\section{The Connection to Runge-Kutta Methods}\label{Sec:Runge-Kutta_Analogy}
%

In this section it is shown that SBP/GSBP time-marching methods can be rewritten as $n$-stage implicit Runge-Kutta methods. The Runge-Kutta connection is then applied to derive some additional accuracy and stability results for nonlinear IVPs. Specifically, diagonal-norm GSBP time-marching methods are shown by definition to be of order $p\ge \min(2q+1,\tau)$ for nonlinear problems. For dense-norm schemes this is reduced to $p\ge \min(q+1,\tau)$. However, these are only guaranteed minimums; the Runge-Kutta order conditions can also be used to supersede these results, as shown in Section \ref{Sec:Methods}. Finally, conditions are derived in this section under which BN-stable dense-norm GSBP time-marching methods can be constructed.

While SBP/GSBP time-marching method form a subset of implicit Runge-Kutta methods, the SBP/GSBP characterization remains important. Firstly, it greatly simplifies the construction of high order fully-implicit time-marching methods with a particular set of characteristics (See \cite{DBZ:2014} and Section \ref{Sec:Methods}). All that is required is a distribution of solution points or a quadrature rule. The resulting scheme is by definition L-stable and yields superconvergence of integral functionals. Furthermore, if the quadrature rule is positive, a diagonal-norm scheme exists and is BN-stable. The SBP/GSBP characterization can also facilitate the analysis of fully-discrete approximations of PDEs. Finally, the generalization to multi-dimensional GSBP operators proposed in \cite{Hicken:2015} may enable space-time discretizations of PDEs with the aforementioned properties. In this case, the explicit connection to Runge-Kutta methods is often lost.

To show the connection to Runge-Kutta methods, consider the $m^{th}$ application of a GSBP time-marching method to the nonlinear IVP \eqref{Eq:NonIVP}: $\mathcal{Y}^\prime = \mathcal{F}(\mathcal{Y},t)$:
\begin{equation}\label{Eq:GLM}
	D \mathbf{y}^{[m]}_\mathrm{d} = H^{-1}\Theta \mathbf{y}^{[m]}_\mathrm{d} = 
	\mathbf{f}^{[m]}_\mathrm{d} - 
	H^{-1}\chi_{t_0}\left(\chi_{t_0}^T \mathbf{y}^{[m]}_\mathrm{d} - 
	\chi_{t_f}^T \mathbf{y}^{[m-1]}_\mathrm{d}
	\right).
\end{equation}
Rearranging \eqref{Eq:GLM} for the solution values using Condition \ref{Ass:Rinv}: $(\Theta+\chi_{t_0}\chi_{t_0}^T)$ has strictly positive eigenvalues and is therefore invertible, and Lemma \ref{Lem:U}: $(\Theta +  \chi_{t_0}\chi_{t_0}^T)^{-1 }\chi_{t_0}=\mathds{1}$, gives
\begin{equation}\label{Eq:GSBPRunge-Kutta1}
	\mathbf{y}^{[m]}_\mathrm{d} = 
	\mathds{1}\tilde{y}^{[m-1]} + 
	\left(\Theta + 	\chi_{t_0}\chi_{t_0}^T\right)^{-1 } H \mathbf{f}^{[m]}_\mathrm{d},
\end{equation}
where $\tilde{y}^{[m-1]} =  \chi_{t_f}^T \mathbf{y}^{[m-1]}_\mathrm{d}$. Projecting \eqref{Eq:GSBPRunge-Kutta1} to the end of the time step and simplifying using Condition \ref{Ass:E}: $\chi_{t_f}^T\mathbf{t}^i = t_f^i$ for $i\in[0,r\ge q\ge 1]$, and Lemma \ref{Lem:B}: $\chi_{t_f}^T(\Theta +  \chi_{t_0}\chi_{t_0}^T)^{-1}=\mathds{1}^T$, yields:
\begin{equation}\label{Eq:GSBPRunge-Kutta2}
	\tilde{y}^{[m]} = \tilde{y}^{[m-1]} + \mathds{1}^T H \mathbf{f}^{[m]}_\mathrm{d}.
\end{equation}
These equations describe a set of intermediate values $\mathbf{y}^{[m]}_\mathrm{d}$, constructed from a single initial value $\tilde{y}^{[m-1]}$. In turn, these values are used to generate a solution one step forward in time $\tilde{y}^{[m]}$. This is equivalent to a Runge-Kutta scheme written in the form
\begin{equation}\label{Eq:Runge-Kutta1}
	\tilde{y}^{[m]} = \tilde{y}^{[m-1]} + h\sum_{j=1}^n b_{j}\mathcal{F}(y_{\mathrm{d},j},t^{[m-1]}+c_jh),
\end{equation}
with internal stage approximations:
\begin{equation}\label{Eq:Runge-Kutta2}
	y_k = \tilde{y}^{[m-1]} + 
		h\sum_{j=1}^n A_{kj}\mathcal{F}(y_{\mathrm{d},j},t^{[m-1]}+c_jh)  \quad \mbox{for } k=1,\ldots ,n,
\end{equation}
where $A_{kj}$ and $b_{j}$ are the coefficients of the method with abscissa $\mathbf{c}$, and $h=t^{[m]}_f-t^{[m]}_0$ is the step size. Comparing the two sets of equations, the Runge-Kutta coefficient matrices associated with a SBP/GSBP time-marching method are:
\begin{equation}\label{Eq:coef1}
	\begin{array}{c}
		A = \frac{1}{h} \left(\Theta +  \chi_{t_0}\chi_{t_0}^T\right)^{-1 }H, \\\\
		b^T = \chi_{t_f}^T A = \frac{1}{h} \mathds{1}^T H.
	\end{array}
\end{equation}
The factor of $1/h$ in \eqref{Eq:coef1} stems from the fact that for SBP/GSBP time-marching, the step size is implicitly defined in the norm. Similarly, the abscissa of the Runge-Kutta characterization must be rescaled and translated from $[t^{[m]}_0,t_f^{[m]}]$ to $[0,1]$:
\begin{equation}\label{Eq:GLM3}
	\mathbf{c} = \frac{\mathbf{t} - \mathds{1}t^{[m]}_0}{h}.
\end{equation}
The application of GSBP and projection operators applied to \eqref{Eq:GLM3} is discussed in Appendix \ref{App:Abscissa}. These relationships greatly simplify the analysis in the subsequent sections.

%
\subsection{Accuracy for Nonlinear Initial Value Problems}
%

In Section \ref{Sec:Acc} we have shown that the solution at the end of each time step approximated from the numerical solution from a GSBP time-marching method is superconvergent. However, the theory of superconvergence is limited to IVPs which are linear with respect to the solution. In this section, we use the connection to Runge-Kutta methods to generate general accuracy results for fully nonlinear problems. This is accomplished by comparing the conditions imposed on GSBP time-marching methods and the simplifying order conditions derived for Runge-Kutta methods. 

A brief discussion of the full nonlinear order conditions for Runge-Kutta methods is also presented, which can be used to supersede the aforementioned minimum guaranteed order results.

%
\subsubsection{Simplifying assumptions}\label{Sec:Simple}
%

This section presents a brief introduction to the simplifying order conditions derived for Runge-Kutta methods. The Runge-Kutta characterization of GSBP time-marching methods is then substituted into these conditions to derive some general accuracy results.

The full Runge-Kutta order conditions for general nonlinear IVPs become increasingly difficult to solve as the order increases. To ease the construction of higher-order Runge-Kutta schemes, a simpler set of sufficient conditions were derived in \cite{Butcher:1964}. This simplified set of equations are referred to as simplifying order conditions and are summarized with the following theorem presented without proof (See Theorem 7 in \cite{Butcher:1964}):\\

\begin{theorem}\label{Thm:Simp}
	If the coefficients $A$, $b^T$, and $\mathbf{c}$ of a Runge-Kutta method satisfy the following conditions:
	\begin{equation}\label{Eq:Simp1}
		B(p):\quad	b^T\mathbf{c}^{j-1} = \frac{1}{j}, \quad j \in [1,p]
	\end{equation}
	\begin{equation}\label{Eq:Simp2}
		C(\hat{q}):\quad	A\mathbf{c}^{j-1} = \frac{\mathbf{c}^j}{j}, \quad j \in [1,\hat{q}]
	\end{equation}
	\begin{equation}\label{Eq:Simp3}
		D(\xi):\quad	A^T B_d\mathbf{c}^{j-1} = \frac{1}{j}B_d(\mathds{1}-\mathbf{c}^j), \quad j \in [1,\xi]
	\end{equation}
	with $p\le 2\hat{q}+2$ and $p\le \hat{q}+\xi+1$, where $B_d$ is a diagonal matrix formed by the entries of $b$, then the Runge-Kutta method will be of order $p$.
\end{theorem}

\ \\[-4.5ex]

The first simplifying condition $B(p)$ \eqref{Eq:Simp1} is the requirement that $b^T$ be a quadrature rule of order $p$. The relationship between SBP/GSBP operators and quadrature rules was examined in \cite{Hicken:2011c,DBZ:2014}. The following lemma shows that GSBP time-marching methods satisfy the first simplifying order condition $B(p)$ for $p=\tau$, where $\tau$ is the order of the associated quadrature rule:\\

\begin{lemma}\label{Lem:Bp}
	SBP and GSBP time-marching methods with a norm matrix associated with a quadrature rule $\mathbf{w} = H\mathds{1}$ of order $\tau$, satisfy condition $B(p)$ for $p=\tau$.
\end{lemma}

\

\begin{proof}
	 Beginning with first simplifying condition $B(p)$ \eqref{Eq:Simp1} and substituting \eqref{Eq:ab4}: $\mathbf{c}^p = \frac{1}{h^p} \sum_{i=0}^p \binom{p}{i}(\mathbf{t}^{[m]})^{p-i}(-t_0^{[m]})^i$, gives:	
	\begin{equation}
		b^T\mathbf{c}^{j-1} = \frac{1}{h}\mathbf{w}^T\mathbf{c}^{j-1} = \frac{1}{h^{j}} \sum_{i=0}^{j-1} \binom{{j-1}}{i}\mathbf{w}^T(\mathbf{t}^{[m]})^{{j-1}-i}(-t_0^{[m]})^i, \quad j \in [1,p].
	\end{equation}
	Integrating the distribution of solution points $\mathbf{t}^{[m]}$ using the quadrature $\mathbf{w}$ yields:
	\begin{equation}
		b^T\mathbf{c}^{j-1} = \frac{1}{h^{j}} \sum_{i=0}^{j-1} \binom{{j-1}}{i}\frac{1}{{j-i}}((t_f^{[m]})^{j-i}-(t_0^{[m]})^{j-i})(-t_0^{[m]})^i, \quad j \in [1,p\le\tau].
	\end{equation}
	Using the identity $\binom{{j-1}}{i} = \binom{{j}}{i}\frac{j-i}{j}$ and simplifying gives:
	\begin{equation}
		b^T\mathbf{c}^{j-1} = \frac{1}{j h^{j}} \sum_{i=0}^{j-1} \binom{{j}}{i}
		\left(		(t_f^{[m]})^{j-i}(-t_0^{[m]})^i -(t_0^{[m]})^{j}(-1)^i		\right), \quad j \in [1,p\le\tau].
	\end{equation}
	Next, we raise the upper bound on the summation by adding and subtracting the argument of the summation with index $i=j$:
	\begin{align} \label{Eq:Bp1}
		b^T\mathbf{c}^{j-1} &= 	\frac{1}{j h^{j}} \sum_{i=0}^{j} \binom{{j}}{i}
		\left(		(t_f^{[m]})^{j-i}(-t_0^{[m]})^i -(t_0^{[m]})^{j}(-1)^i		\right)\\[1ex]
		&\qquad\qquad -\left(		(t_f^{[m]})^{j-i}(-t_0^{[m]})^i -(t_0^{[m]})^{j}(-1)^i	 		\right)\bigg|_{i=j}, \quad j \in [1,p\le\tau].	\nonumber
	\end{align}
	The summation can be simplified using the identity $\sum_{i=0}^{j}\binom{{j}}{i}(-1)^i=0$:
	\begin{equation}
		\sum_{i=0}^{j} \binom{{j}}{i}(t_f^{[m]})^{j-i}(-t_0^{[m]})^i - (t_0^{[m]})^{j} \sum_{i=0}^{j} \binom{{j}}{i}(-1)^i	
		= \sum_{i=0}^{j} \binom{{j}}{i}(t_f^{[m]})^{j-i}(-t_0^{[m]})^i,
	\end{equation}
	which is equivalent to $h^j = (t_f^{[m]}-t_0^{[m]})^j$. Therefore, the first term in \eqref{Eq:Bp1} simplifies to $\frac{h^j}{jh^j}=1/j$. The second term in \eqref{Eq:Bp1} simplifies to zero:
	\begin{equation}
		- \left(		(t_f^{[m]})^{0}(-t_0^{[m]})^j -(t_0^{[m]})^{j}(-1)^j	 		\right) = 
		- \left(		(t_0^{[m]})^{j}(-1)^j - (t_0^{[m]})^{j}(-1)^j	 		\right) = 0.
	\end{equation}
	Thus, $b^T\mathbf{c}^{j-1} = 1/j$ for $j \in [1,p\le\tau],$ completing the proof.
\end{proof}

\ \\[-4.5ex]

The simplifying conditions $C(\hat{q})$ \eqref{Eq:Simp2} are known as the stage-order conditions. It describes the order to which the intermediate stage values \eqref{Eq:Runge-Kutta2} approximate $\mathcal{Y}(t_0 + c_i)$. This property influences the convergence of problems with stiff source terms \cite{Prothero,Nordstrom:2014} as well as singular perturbation problems (\textit{e.g.} \cite{Hairer:1988}). For SBP and GSBP time-marching methods these are the solution values $\mathbf{y}_\mathrm{d}$. The following lemma states that the second simplifying condition $C(\hat{q})$ is satisfied with $\hat{q}\ge q$, where $q$ is the order of the GSBP operator:\\

\begin{lemma}\label{Lem:Cq}
	SBP and GSBP time-marching methods of order $q$, satisfy the stage order condition $C(\hat{q})$ for $\hat{q} \ge q$.
\end{lemma}

\

\begin{proof}
	Consider the stage order conditions \eqref{Eq:Simp2}:
	\begin{equation}
		A\mathbf{c}^{j-1} =\frac{1}{h}(\Theta + \chi_{t_0}\chi_{t_0}^T)^{-1 } H\mathbf{c}^{j-1} = \frac{1}{j}\mathbf{c}^{j},\text{ for } j\in[1,\hat{q}].
	\end{equation}
	Multiplying through by $h H^{-1 } (\Theta + \chi_{t_0}\chi_{t_0}^T)$ and simplifying yields
	\begin{equation}	
			\mathbf{c}^{j-1} = \frac{1}{j}hD\mathbf{c}^{j} + \frac{1}{j}h H^{-1}\chi_{t_0}\chi_{t_0}^T\mathbf{c}^{j},\text{ for } j\in[1,\hat{q}].
	\end{equation}
	If $\hat{q}\le q$, then by Lemma \ref{Lem:AbDer}: $D\mathbf{c}^{j} = \frac{1}{h}j\mathbf{c}^{j-1}$ for $j\in[0,q]$, the first term on the right-hand side becomes $\mathbf{c}^{j-1}$, and by Lemma \ref{Lem:Ab0}: $\chi_{t_0}^T\mathbf{c}^{j} = 0$ for $j\in[0,r\ge q]$, the second term becomes $0$. Thus, an identity is recovered, proving the theorem.
\end{proof}

\ \\[-4.5ex]

The final simplifying condition $D(\xi)$ \eqref{Eq:Simp3} defines a relationship between certain order conditions based on their form. The following lemma states that diagonal-norm GSBP time-marching methods satisfy this condition with $\xi\ge q$:\\

\begin{lemma}\label{Lem:Dxi}
	Diagonal-norm SBP and GSBP time-marching methods of order $q$, satisfy condition $D(\xi)$ for $\xi\ge q$.
\end{lemma}

\

\begin{proof} 
	Consider the third simplifying condition $D(\xi)$:
	\begin{equation}\label{Eq:Simp4}
		A^T B_d\mathbf{c}^{j-1} = \left(\frac{1}{h} \left(\Theta +  \chi_{t_0}\chi_{t_0}^T\right)^{-1 }H\right)^T B_d\mathbf{c}^{j-1} = \frac{1}{j}B_d(\mathds{1}-\mathbf{c}^j), \quad j \in [1,\xi].
	\end{equation}
	For diagonal-norm GSBP time-marching methods, $B_d$ is equivalent to the norm $H$. Making this substitution, multiplying through by $H^{-1 } \left(\Theta +  \chi_{t_0}\chi_{t_0}^T\right)^T H^{-1 } $, and simplifying yields:
	\begin{equation}
		\frac{1}{h} \mathbf{c}^{j-1} = \frac{1}{j} H^{-1} \left(\Theta^T +  \chi_{t_0}\chi_{t_0}^T\right)  (\mathds{1}-\mathbf{c}^j), \quad j \in [1,\xi].
	\end{equation}
	Using Condition \ref{Ass:E}: $\Theta+\Theta^T = \chi_{t_f}\chi_{t_f}^T - \chi_{t_0}\chi_{t_0}^T$, and expanding gives:
	\begin{equation}
		\frac{1}{h} \mathbf{c}^{j-1} = \frac{1}{j} \left(-H^{-1} \Theta(\mathds{1}-\mathbf{c}^j) +  H^{-1} \chi_{t_f}\chi_{t_f}^T(\mathds{1}-\mathbf{c}^j)\right),  \quad j \in [1,\xi].
	\end{equation}
	If $\xi\le q$, then by Lemma \ref{Lem:AbDer}: $D\mathbf{c}^{j} = \frac{1}{h}j\mathbf{c}^{j-1}$ for $j\in[0,q]$, the first term on the left-hand side reduces to $\frac{1}{h} \mathbf{c}^{j-1}$, and by Lemma \ref{Lem:Abf}: $\chi_{t_f}^T\mathbf{c}^j =1$ for $j\in[0,r\ge q]$, the second term simplifies to zero. Thus we recover an identity for $\xi\le q$, proving the lemma.
\end{proof}

\ \\
Combining this result with Lemmas \ref{Lem:Bp} and \ref{Lem:Cq}, the order of diagonal-norm SBP and GSBP time-marching methods follows immediately, summarized by the following theorem:\\

\begin{theorem}
	Diagonal-norm SBP and GSBP time-marching methods of order $p\ge \min(\tau,2q+1)$.
\end{theorem}

\

\begin{proof}
	The result follows from Theorem \ref{Thm:Simp}, as well as Lemmas \ref{Lem:Bp}, \ref{Lem:Cq}, and \ref{Lem:Dxi}.
\end{proof}

\ \\[-4.5ex]

Therefore, the superconvergence theory derived for diagonal-norm SBP and GSBP time-marching methods in Theorem \ref{Thm:SupSol}, also holds in the nonlinear case, $p=\min(\tau,2q+1)$.

In contrast, dense-norm SBP and GSBP time-marching methods do not in general satisfy the third simplifying order condition $D(\xi)$ \eqref{Eq:Simp3} with $\xi>0$. Therefore without the support of the third condition $D(\xi>0)$, the maximum guaranteed order of dense-norm SBP and GSBP time-marching methods is $\min(\tau,q+1)$, summarized in the following theorem:\\

\begin{theorem}
	Dense-norm SBP and GSBP time-marching methods are of order $p\ge \min(\tau,q+1)$.
\end{theorem}

\

\begin{proof}
	The result follows from Theorem \ref{Thm:Simp}, as well as Lemmas \ref{Lem:Bp} and \ref{Lem:Cq}.
\end{proof}

\ \\[-1.5ex]
This is not in general as high as the superconvergence shown in Theorem \ref{Thm:SupSol}; however the third simplifying condition can be used to guide the construction of higher-order dense-norm GSBP time-marching methods.

In summary, we have applied the connection to Runge-Kutta methods to extend the superconvergence theory of the solution approximated at the end of each time step to fully nonlinear problems. For diagonal-norm GSBP time-marching methods, the rate of superconvergence is identical to what is shown in Section \ref{Sec:Super} for problems linear with respect to the solution. For dense-norm GSBP time-marching methods, the bounds are more restrictive than the previous result. Finally, the simplifying order conditions can be used to guide the construction of GSBP time-marching methods which supersede these minimum guaranteed results. An example is discussed in Section \ref{Sec:Methods}.

%
\subsubsection{Full order conditions}\label{Sec:FOCon}
%

The full order conditions for a Runge-Kutta scheme can be found in several references, for example \cite{Hairer:2000,Butcher:Book}. For reference, the conditions for orders one through four, assuming stage consistency: $A\mathds{1} = \mathbf{c}$, are:
\begin{equation}\label{Eq:FRKOC}
	\begin{array}{ccc}
		b^T \mathds{1} = 1 				& \hspace{2cm}	& b^T \mathbf{c}^3 = \frac{1}{4} \\[1ex]
		b^T \mathbf{c} = \frac{1}{2} 	& \hspace{2cm}	& b^T A \mathbf{c}^2 = \frac{1}{8} \\[1ex]
		b^T \mathbf{c}^2 = \frac{1}{3}	& \hspace{2cm}	& b^T C A \mathbf{c} = \frac{1}{12} \\[1ex]
		b^T A \mathbf{c} = \frac{1}{6} & \hspace{2cm}	& b^T A A \mathbf{c} = \frac{1}{24}				
	\end{array}
\end{equation}
where $C = \mathrm{diag}(\mathbf{c})$. Given the relationship between GSBP time-marching methods and the Runge-Kutta coefficient matrices derived above, one can simply insert these relationships into the algebraic order conditions and solve for the coefficients. These systems of equations are not necessarily easy to solve, but can be exploited in the construction of GSBP time-marching methods. For example, diagonally-implicit schemes are limited to stage-order $\hat{q}=1$ due to their form. Using the GSBP approach, the maximum guaranteed order of the resulting time-marching method is $p = 2q+1=3$. There is no theoretical way to determine if the scheme is of higher order, or to guide the construction of a higher-order scheme if it is not. The Runge-Kutta order conditions fill this role. In Section \ref{Sec:Methods} the full order conditions are applied to guide the construction of a diagonally implicit GSBP time-marching method with stage order $1$ and global order $4$.

\subsection{Nonlinear Stability of Dense-Norm GSBP Time-Marching Methods}

Dense-norm SBP and GSBP time-marching methods are not in general BN-stable or energy stable (See Section \ref{Sec:Stab}). In this section, however, we show that it is possible to construct BN-stable dense-norm GSBP time-marching methods. This is accomplished by considering the algebraic conditions for BN-stability derived in \cite{Butcher:1975,Crouzeix:1979}:\\

\begin{theorem}
A Runge-Kutta method is BN-stable if:\\[-1.5ex]
\begin{enumerate}[i)]
	\item $b_i \ge 0$ for $i=1,\ldots,s$;
	\item $A$ is invertible; and
	\item $\widehat{M} = B_dA^{-1} + (A^{-1})^TB_d - (A^{-1})^T b b^TA^{-1}$ is non-negative definite,
\end{enumerate}

\ \\[-1.5ex]
where $B_d$ is a diagonal matrix formed by the elements of $b$.
\end{theorem}

\ \\[-1.5ex]
By substituting the Runge-Kutta characterization of GSBP time-marching methods into this theorem, we obtain the conditions under which dense-norm GSBP time-marching methods can be constructed:\\

\begin{theorem}\label{Thm:Alg}
	A GSBP time-marching method which satisfies Condition \ref{Ass:Rinv} is BN-stable if it associated with a non-negative quadrature rule $\mathbf{w} = H\mathds{1}$, and the coefficient matrices satisfy:
	\begin{equation}\label{Eq:Stab1}
		W_dH^{-1}\left(\Theta+\chi_{t_0}\chi_{t_0}^T\right) + \left(\Theta+\chi_{t_0}\chi_{t_0}^T\right)^TH^{-1}W_d - \chi_{t_f}\chi_{t_f}^T,
	\end{equation}
	non-negative definite, where $W_d$ is a diagonal matrix formed by the elements of $\mathbf{w} = H\mathds{1}$.
\end{theorem}

\

\begin{proof}
	Condition i): The $b$ coefficient matrix is equal to $\frac{1}{h}\mathbf{w}$. Therefore, if the weights of the associated quadrature rule are non-negative: $w_i \ge 0$, so will the entries of the $b$ coefficient matrix. This is immediately satisfied for diagonal-norm GSBP time-marching methods by Definition \ref{Def:SBPGen}: $H$ must be a SPD matrix.\\

	Condition ii): The $A$ coefficient matrix is equal to $\frac{1}{h}(\Theta+\chi_{t_0}\chi_{t_0}^T)^{-1}H$. The matrix $(\Theta+\chi_{t_0}\chi_{t_0}^T)$ is invertible by Condition \ref{Ass:Rinv}, and the norm $H$ is SPD by Definition \ref{Def:SBPGen}.  Therefore the $A$ coefficient matrix of a GSBP time-marching method is invertible by construction.\\
	
	Condition iii): To begin, substitute the relationship $b^T = \chi_{t_f}^TA$ into $\widehat{M}$:
	\begin{equation}\label{Eq:denseAlg}
		\widehat{M} = B_dA^{-1} + (A^{-1})^TB_d - (A^{-1})^T A^T\chi_{t_f} \chi_{t_f}^T A A^{-1}.
	\end{equation}
	Next, consider that the $b$ coefficient matrix is equal to $\frac{1}{h}\mathbf{w}$, and therefore $B_d$ is equal to $\frac{1}{h}\diag(\mathbf{w})$ or $\frac{1}{h}W_d$. Substituting this relationship into \eqref{Eq:denseAlg}, as well as the characterization of the $A$ coefficient matrix $\frac{1}{h}(\Theta+\chi_{t_0}\chi_{t_0}^T)^{-1}H$, and simplifying yields:
	\begin{equation}\label{Eq:denseAlg1}
		\widehat{M} = W_dH^{-1}\left(\Theta+\chi_{t_0}\chi_{t_0}^T\right) + \left(\Theta+\chi_{t_0}\chi_{t_0}^T\right)^TH^{-1}W_d - \chi_{t_f}\chi_{t_f}^T.
	\end{equation}
	Therefore, if \eqref{Eq:Stab1} is non-negative definite, so will $\widehat{M}$. For diagonal-norm GSBP time-marching methods the matrix $W_d$ is equal to the norm $H$. In this case \eqref{Eq:denseAlg1} simplifies to $\widehat{M} = \chi_{t_0}\chi_{t_0}^T$, which is always non-negative definite and satisfies the condition.
\end{proof}

\ \\[-4.5ex]

Thus, using the connection to Runge-Kutta methods, we have derived the conditions under which dense-norm GSBP time-marching methods can be constructed. In the proof, we have shown that all diagonal norm GSBP time-marching methods are BN-stable, which is consistent with Theorem \ref{Thm:EngyStb_diag}.

We can take these results one step further. It is well known that for non-confluent Runge-Kutta schemes (schemes with unique $c_i$) the criteria for algebraic-stability, BN-stability, B-stability, and AN-stability are all equivalent \cite{HairerII}. Since all dual-consistent SBP and GSBP time-marching methods are non-confluent by Definition \ref{Def:SBPGen}, those which satisfy Theorem \ref{Thm:Alg} will also have these additional stability properties.

%
\section{Examples of GSBP Time-Marching Methods}\label{Sec:Methods}
%

This section applies the theory developed in this article to construct some known and novel Runge-Kutta schemes which are based on GSBP operators.

%
\subsection{Lobatto IIIC discontinuous-collocation Runge-Kutta methods}\label{Sec:LGL}
%

A set of diagonal-norm GSBP operators which lead to a known class of Runge-Kutta methods, are those based on Gauss-Lobatto points. These operators are also spectral-element operators and where considered with the use of SATs in \cite{Gassner2013}. As an example, consider the four-node Gauss-Lobatto points in the domain $[-1,1]$:
\begin{equation}
	\mathbf{t} = \left[
	\begin{array}{cccc}
	-1 &  - \tfrac{1}{5}\sqrt{5} & \tfrac{1}{5}\sqrt{5} & 1
	\end{array}
	\right]^T.
\end{equation}
The corresponding GSBP norm, whose entries are the Gauss-Lobatto quadrature weights, and GSBP operator are:
\begin{equation}\label{Eq:LGL1}
	H = \left[
	\begin{array}{cccc}
	\tfrac{1}{6} \\[1ex]
	& \tfrac{5}{6} \\[1ex]
	& & \tfrac{5}{6} \\[1ex]
	& & & 	\tfrac{1}{6} 
	\end{array}
	\right], \quad
	D = \left[
	\begin{array}{cccc}
	-3 & -\tfrac{5\sqrt{5}}{\sqrt{5}-5} & -\tfrac{5\sqrt{5}}{\sqrt{5}+5} & \frac{1}{2}\\[1ex]
	\tfrac{\sqrt{5}}{\sqrt{5}-5} & 0 & \frac{\sqrt{5}}{2} & -\tfrac{5\sqrt{5}}{\sqrt{5}+5}\\[1ex]
	\tfrac{\sqrt{5}}{\sqrt{5}+5} & \frac{-\sqrt{5}}{2} & 0 & -\tfrac{5\sqrt{5}}{\sqrt{5}-5}\\[1ex]
	-\frac{1}{2} & \tfrac{5\sqrt{5}}{\sqrt{5}+5} & \tfrac{5\sqrt{5}}{\sqrt{5}-5} & 3
	\end{array}
	\right].
\end{equation}
This operator is associated with the exact projection operators $\chi_{-1}=[1,0,\ldots,0]^T$ and $\chi_{1}=[0,\ldots,0,1]^T$. Applying the Runge-Kutta characterization derived in Section \ref{Sec:Runge-Kutta_Analogy}, the coefficients of the equivalent Runge-Kutta scheme are
\begin{equation}\label{Eq:LGL2}
	A= \left[ \begin {array}{cccc} 
	\frac{1}{12}
	 & -\frac{-\sqrt{5}}{12}
	 & \frac{-\sqrt{5}}{12}
	 & -\frac{1}{12}\\ \noalign{\medskip}
	 
	 \frac{1}{12}
	 &\frac{1}{4}
	 &{\frac {10-7\sqrt {5}}{60}}
	 &{\frac {\sqrt {5}}{60}}\\ \noalign{\medskip}
	 
	 \frac{1}{12}
	 &{\frac {10+7\sqrt {5}}{60}}
	 &\frac{1}{4}
	 &-{\frac {\sqrt {5}}{60}}\\ \noalign{\medskip}
	 
	 \frac{1}{12}
	 &{\frac {5}{12}}
	 &{\frac {5}{12}}
	 &\frac{1}{12}\end {array} \right],
\end{equation}
and
\begin{equation}\label{Eq:LGL3}
	b = \left[
	\begin{array}{cccc}
	\tfrac{1}{12} 
	& \tfrac{5}{12}
	& \tfrac{5}{12}
	& \tfrac{1}{12} 
	\end{array}
	\right],
\end{equation}
with abscissa:
\begin{equation}
	\mathbf{c} = \left[
	\begin{array}{cccc}
	0 & \tfrac{1}{2} - \tfrac{1}{10}\sqrt{5} & \tfrac{1}{2} + \tfrac{1}{10}\sqrt{5} & 1
	\end{array}
	\right]^T.
\end{equation}
This is the four-point Lobatto IIIC discontinuous-collocation Runge-Kutta scheme \cite{HairerIII,Chipman:1971,Ehle:1969,Axelsson:1972}. A similar approach can be applied on Radau quadrature points, leading to the Radau IA and IIA  discontinuous-collocation Runge-Kutta schemes. These scheme are all L-stable, BN-stable, and energy stable. Interestingly, the Radau IIA scheme also satisfies the stage-order conditions $\hat{q}= q+1$, one order greater than guaranteed by the GSBP theory (Lemma \ref{Lem:Cq}).

%
\subsection{Gauss GSBP time-marching methods}\label{Sec:LG}
%

The generalized framework cannot be used in the same way to construct the well-known order $2n$ Gauss collocation Runge-Kutta methods of Butcher \cite{Butcher:1964} and Kuntzmann \cite{Kuntzmann:1961}. The reason is that these schemes are not L-stable, which all GSBP time-marching methods are. However, using the GSBP approach we can construct a class of time-marching methods of order $2n-1$ on the Gauss points which are L-stable. Consider the four-node Gauss points in the interval $[-1,1]$, shown here to sixteen decimal places:
\begin{equation}
{\tiny
	\mathbf{t} = \left[
	\begin{array}{cccc}
	-0.8611363115940526 &
	-0.3399810435848563 &
	 0.3399810435848563 &
	 0.8611363115940526
	\end{array}
	\right]^T.
}
\end{equation}
The corresponding norm, whose entries are the Gauss quadrature weights, and resulting GSBP operator are:
\begin{equation}
{\tiny
	H= \left[ 
	\begin {array}{cccc}
	0.3478548451374539 & & & \\ [1ex]
	 & 0.6521451548625461 & & \\ [1ex]
	 & & 0.6521451548625461 & \\ [1ex]
	 & & & 0.3478548451374539
	\end{array}
	\right], 
}
\end{equation}
\begin{equation}
{\tiny
	D = \left[
	\begin{array}{cccc}
	-3.3320002363522817 &  4.8601544156851962 & -2.1087823484951789 &  0.5806281691622644 \\[1ex]
	-0.7575576147992339 & -0.3844143922232086 &  1.4706702312807167 & -0.3286982242582743 \\[1ex]
	 0.3286982242582743 & -1.4706702312807167 &  0.3844143922232086 &  0.7575576147992339 \\[1ex]
	-0.5806281691622644 &  2.1087823484951789 & -4.8601544156851962 &  3.3320002363522817
	\end{array}
	\right],
}
\end{equation}
with projection operators
\begin{equation}
{\tiny
\chi_{-1}=\left[
	\begin{array}{cccc}
	1.5267881254572668 & -0.8136324494869273 &  0.4007615203116504 & -0.1139171962819899
	\end{array}
	\right]^T
}
\end{equation}
and
\begin{equation}
{\tiny
\chi_{1}=\left[
	\begin{array}{cccc}
	 -0.1139171962819899 & 	0.4007615203116504 & -0.8136324494869273 & 1.5267881254572668 
	\end{array}
	\right]^T
}
\end{equation}
These operators can also be derived using the spectral-element approach. Applying the Runge-Kutta characterization, the coefficients of the equivalent Runge-Kutta scheme are
\begin{equation}
{\tiny
	A= \left[ \begin {array}{cccc} 
	0.0950400941860569 & -0.0470608105772507 &  0.0330840931816566 & -0.0116315325874891 \\[1ex]
	0.1772065313616314 &  0.1906741915282288 & -0.0555183314150631 &  0.0176470867327749 \\[1ex]
	0.1781035081124255 &  0.3263151032211517 &  0.1906741915282288 & -0.0251022810693778 \\[1ex]
	0.1694061893528291 &  0.3339017452341202 &  0.3322201270240200 &  0.0950400941860569
	\end {array} \right],
}
\end{equation}
and
\begin{equation}
{\tiny
	b = \left[
	\begin{array}{cccc}
	0.0869637112843635 & 	0.1630362887156365 & 	0.1630362887156365 & 	0.0869637112843635
	\end{array}
	\right],
}
\end{equation}
with abscissa:
\begin{equation}
{\tiny
	\mathbf{c} = \left[
	\begin{array}{cccc}
	0.0694318442029737 & 	0.3300094782075719 & 	0.6699905217924281 & 	0.9305681557970263
	\end{array}
	\right]^T.
}
\end{equation}
The abscissa and solution update vector $b$ are equivalent to the eighth order method of Butcher and Kuntzmann; only the $A$ coefficient matrix differs. The GSBP time-marching method is seventh-order, one order lower than the method of Butcher and Kuntzmann, but is L-stable. Using the GSBP approach, other methods in this class are very straight forward to derive.

%
\subsection{Diagonally-implicit GSBP methods}\label{Sec:DIGSBP}
%

In this section, diagonally-im\-plic\-it Runge-Kutta schemes are constructed which are based on GSBP operators. Diagonally-im\-plic\-it methods are often more efficient than fully-implicit schemes, especially in terms of memory usage, and are therefore of particular interest. For these examples, coefficients of the GSBP operator are first constrained such that the resulting Runge-Kutta scheme is diagonally-implicit and satisfies the minimal requirements of Definition \ref{Def:SBPGen}. The former is done by using the fact that the inverse of a lower triangular matrix is also lower triangular. Therefore, decomposing $\Theta$ into symmetric $\Theta_S=\frac{1}{2}\tilde{E}$ and anti-sym\-met\-ric components $\Theta_A$ (See \textit{e.g.} \cite{DBZ:2014}), the coefficients of $\Theta_A$ are chosen such that $H^{-1}(\Theta+\chi_{t_0}\chi_{t_0}^T)$, the inverse of the coefficient matrix $A$, is lower triangular. The remaining coefficients in the GSBP operator and corresponding Runge-Kutta scheme, including the distribution of solution points and weights of the associated quadrature, are solved for using the full Runge-Kutta order conditions.

The first example is a novel three-stage third-order diagonal-norm GSBP scheme. Several coefficients are determined by solving the order conditions \eqref{Eq:FRKOC}. The remaining free coefficients are chosen to minimize the L$_2$-norm of the fourth-order conditions \cite{Prince:1981}. The abscissa of the GSBP operator to sixteen decimal places is:
\begin{equation}
{\tiny
\mathbf{t}= \left[ \begin {array}{ccc}   
 0.0585104413419415 & 
 0.8064574322792799 & 
 0.2834542075672883
 \end {array}
 \right]^T,
 }
\end{equation}
which is already chosen to be for the domain $[0,1]$. Likewise, the norm is determined to be:
\begin{equation}
{\tiny
H= \left[ \begin {array}{ccc}  
0.1008717264855379 & &\\ [1ex]
 & 0.4574278841698629 & \\ [1ex]
 & & 0.4417003893445992
\end {array} 
\right],
}
\end{equation}
which are the weights of the associated quadrature rule, and the GSBP derivative operator is:
\begin{equation}
{\tiny
D_1= \left[ \begin {array}{ccc} 
-12.3737796851209214 & -3.4099304182988046 & 15.7837101034197260\\ [1ex]
- 1.6186577488308495 &  1.2158491567586837 &  0.4028085920721658\\ [1ex]
- 0.9626808228023090 &  1.4979849320764039 & -0.5353041092740949
\end {array} \right],
}
\end{equation}
with projection operators:
\begin{equation}
{\tiny
\chi_{{0}}= \left[ 
\begin {array}{ccc} 
  1.7239953104443755 &   0.1995165337199744 & - 0.9235118441643498
\end {array} \right] ^T,
}
\end{equation}
and
\begin{equation}
{\tiny
\chi_{{1}}= \left[ 
\begin {array}{ccc} 
- 0.6898048930346554 &   1.0733748002069487 &   0.6164300928277068
\end {array} \right]^T.
}
\end{equation}
It is interesting to note that the distribution of solution points is not ordered, $t_i \ngtr t_{i-1}$. This is not uncommon for time-marching methods, however, this restriction is often imposed on GSBP operators for spatial applications (see \textit{e.g.} \cite{DBZ:2014}). Furthermore, unlike the schemes considered in Sections \ref{Sec:LGL} and \ref{Sec:LG} which are constructed from collocated spectral-element operators, these were derived without any connection to basis functions. The additional flexibility enables a third-order method that is diagonally-implicit to be constructed with only three points. The equivalent diagonally-implicit Runge-Kutta scheme has the following coefficient matrices:
\begin{equation}
{\tiny
A= \left[ \begin {array}{ccc} 
0.0585104413426586 & &\\ [1ex]
0.0389225469556698 &  0.7675348853239251 &\\ [1ex]
0.1613387070350185 & -0.5944302919004032 & 0.7165457925008468
\end {array} \right] ,
}
\end{equation}
and
\begin{equation}
{\tiny
b= \left[ \begin {array}{ccc}  
0.1008717264855379 & 
0.4574278841698629 & 
0.4417003893445992
\end {array} \right],
}
\end{equation}
with $\mathbf{c}=\mathbf{t}$. Even though the GSBP derivative operator is dense, the resulting Runge-Kutta scheme is diagonally-implicit. In addition, since the norm associated with the GSBP operator is diagonal, the scheme is by definition L-stable, BN-stable, and energy-stable.

As a second example, a four-stage fourth-order diagonally-implicit GSBP scheme is constructed. This goes beyond the order guaranteed by the GSBP theory alone. The full Runge-Kutta order conditions are used to derive all of the coefficients of the GSBP operators and hence the equivalent Runge-Kutta scheme. The nodal distribution of the scheme determined by solving the full Runge-Kutta order conditions \eqref{Eq:FRKOC} is to sixteen decimal places:
\begin{equation}
{\tiny
\mathbf{t}= \left[ \begin {array}{cccc}
0.5975501145870646 & 
0.1236947892666459 & 
0.9813648784844768 & 
0.2188347157850838 
\end {array} \right] ,
}
\end{equation}
already chosen to be for the domain $[0,1]$. Likewise, the corresponding norm is:
\begin{equation}
{\tiny
H= \left[ \begin {array}{cccc} 
 0.5263633266867775 &  &  &  \\
  & 0.3002573924935185 &  &  \\
  &  & 0.1447678514141155 &  \\
  &  &  & 0.0286114294055885
\end {array} \right] ,
}
\end{equation}
which defines the weights of the associated quadrature rule, and the GSBP derivative operator is:
\begin{equation}
{\tiny
D= \left[ \begin {array}{cccc} 
 0.1993658318073258 & -1.654157580888287  &  1.006020084619771  &   0.4487716644611903 \\[1ex]
-1.648792506689303  & -1.212963928918776  &  1.978966716941006  &   0.8827897186670728 \\[1ex]
 3.217338082860363  & -1.615712813301921  & -0.4880781006041668 &  -1.113547168954275 \\[1ex]
 1.271022350640990  & -0.6382938457303877 &  0.6005231745715582 &  -1.233251679482160
\end {array} \right] ,
}
\end{equation}
with projection operators:
\begin{equation}
{\tiny
\chi_{{0}}= \left[ 
\begin {array}{cccc} 
0.8808689243587871 & 
0.9884420520048577 & 
-0.6011474168414327 & 
-0.2681635595222120
\end {array} \right] ,
}
\end{equation}
and
\begin{equation}
{\tiny
\chi_{{1}}= \left[ 
\begin {array}{cccc} 
0.9928785357819795 & 
-0.4986129934126102 & 
0.4691078563418350 & 
0.03662660128879568
\end {array} \right].
}
\end{equation}
Applying the Runge-Kutta characterization of these operators yields the Runge-Kutta coefficient matrices:
\begin{equation}
{\tiny
 A = \left[ 
\begin {array}{cccc} 
 0.5975501145870646 &  &  & \\[1ex]
-0.3662683378362842 &  0.4899631271029300 &  & \\[1ex]
-0.9122346095222909 &  1.395636663278596 &  0.4979628247281717 & \\[1ex]
 4.870201094711127  & -3.007233691002447 & -2.425297972138512  &  0.7811652842149162
\end {array} \right] ,
}
\end{equation}
and 
\begin{equation}
{\tiny
b^T = \left[ 
\begin {array}{cccc} 
0.5263633266867775 & 
0.3002573924935185 & 
0.1447678514141155 & 
0.02861142940558849
\end {array} \right] ,
}
\end{equation}
with abscissa $\mathbf{c}=\mathbf{t}$. Since this Runge-Kutta scheme is constructed from a diagonal-norm GSBP operator, it is L-stable, linearly-stable, algebraically-stable and energy-stable. It is also interesting to note that fourth-order is the highest order possible for diagonally-implicit Runge-Kutta schemes which are algebraically-stable \cite{Hairer:1980}. Therefore, to construct a diagonally-implicit GSBP scheme of order greater than four, it must not be algebraically stable, and therefore cannot be based on a diagonal-norm GSBP operator.

%
\section{Numerical Examples}\label{Sec:Results}
%

This section examines the efficiency of various fully-implicit time-marching methods based on classical SBP and GSBP operators.  We solve the linear convection equation with unit wave speed and periodic boundary conditions:
\begin{equation}\label{LC2}
\begin{split}
	\frac{\partial \mathcal{U}}{\partial t} = - \frac{\partial \mathcal{U}}{\partial x}, 
	\quad x\in[0,2],\\[1ex]
	\mathcal{U}(t=0,x) = \sin(2\pi x),\\[1ex]
	\mathcal{U}(t,x=0) = \mathcal{U}(t,x=2),
\end{split}
\end{equation}
The spatial derivative is discretized with a $100$-block GSBP-SAT discretization, where each block is a $5$-node operator associated with Legendre-Gauss quadrature \cite{DBZ:2014}. This leads to an IVP of the form:
\begin{equation}\label{eq:LCIVP}
	\frac{d \mathcal{Y}}{d t} = \mathcal{A}\mathcal{Y}, \quad \mathcal{Y}_0 = \sin(2\pi \mathbf{x}),
\end{equation}
where $\mathcal{A}$ is a $500\times500$ matrix associated with the spatial discretization. The exact solution of this IVP  is $\mathcal{Y} = e^{\mathcal{A}t}\mathcal{Y}_0$. Applying a Runge-Kutta time-marching method leads to a linear system of equations, which is stored in Matlab's sparse format. Reverse Cuthill-McKee reordering is applied to the linear system and solved using the backslash operator. The solutions were computed using MATLAB 2013a on a 6-core intel Core i7-3930K processor at 3.2GHz with 32GB of RAM.

A summary of the GSBP and non-GSBP time-marching methods investigated is presented in Table \ref{Tab:Abb} along with their associated properties and abbreviations used hereafter. Note that the Radau IIA schemes have stage order $n$, one higher than the associated GSBP operator, which is limited to $n-1$. All methods were implemented as Runge-Kutta schemes. 

\begin{table}[!t]
	\centering
	\begin{tabular}{l|l|l|c|c|c}
		\multicolumn{6}{l}{Classical SBP Methods}\\[1ex]
		\hline
		&&&&&\\[-1.5ex]
		Quadrature & Norm & Label	& $\hat{q}$ & $p$ & L/BN-stable\\
		\hline
		&&&&&\\[-1.5ex]
		Gregory type		& Diag. \cite{Kreiss1974,Nordstrom:2013,Nordstrom:2014} & FD	& $\frac{n}{4}$	& $\frac{n}{2}$ & Y \ /\ Y\\[0.5ex]
						& Block \cite{Strand1994,Nordstrom:2013,Nordstrom:2014} & FDB	& $\frac{n}{2}-1$	& $\frac{n}{2}$ & Y \ /\ N\\[0.5ex]
									
		\multicolumn{6}{l}{}\\
		\multicolumn{6}{l}{GSBP Methods}\\[1ex]
		\hline
		&&&&&\\[-1.5ex]
		Quadrature & Norm & Label	& $\hat{q}$ & $p$ & L/BN-stable\\
		\hline
		&&&&&\\[-1.5ex]
		Newton-Cotes 	& Diag. \cite{DBZ:2014} & NC		& $\lceil\frac{n}{2}\rceil$	& $2\lceil\frac{n}{2}\rceil$  & Y \ /\ Y\\[0.5ex]
									& Dense \cite{Carpenter1996} & NCD	& $n-1$	& $2\lceil\frac{n}{2}\rceil$ 	 & Y \ /\ N\\[0.5ex]
		Lobatto			& Diag.$^{*\dagger}$ \cite{Gassner2013,Chipman:1971,Ehle:1969,Axelsson:1972} & LGL		& $n-1$	& $2n-2$  & Y \ /\ Y\\[0.5ex]
									& Dense	\cite{Carpenter1996} & LGLD		& $n-1$	& $2n-2$  & Y \ /\ N\\[0.5ex]
		Radau IA			& Diag.$^{*\dagger}$ \cite{DBZ:2014,Ehle:1969,Axelsson:1972} & LGRI		& $n-1$	& $2n-1$  & Y \ /\ Y\\[0.5ex]
		Radau IIA		& Diag.$^{*\dagger}$ \cite{Boom:Thesis,Ehle:1969,Axelsson:1972} & LGRII	& $n$	& $2n-1$  & Y \ /\ Y\\[0.5ex]
		Gauss 			& Diag.$^\dagger$ \cite{DBZ:2014} & LG		& $n-1$	& $2n-1$  & Y \ /\ Y \\[0.5ex]		
			
		\multicolumn{6}{l}{}\\
		\multicolumn{6}{l}{non-SBP Methods}\\[1ex]
		\hline
		&&&&&\\[-1.5ex]
		Quadrature & Norm & Label	& $\hat{q}$ & $p$ & L/BN-stable\\
		\hline
		&&&&&\\[-1.5ex]
		Gauss 			& $^*$ \cite{Butcher:1964,Kuntzmann:1961} & GRK		& $n$	& $2n$  & N \ /\ Y \\[0.5ex]	
		ESDIRK5 			& $^*$ \cite{Boom:Thesis} & ESDIRK5		& $2$	& $5$  & Y \ /\ N \\[0.5ex]

	\end{tabular}
	\caption{Summary of SBP operators, their associated abbreviations and general properties. Notes: 1) the general properties of diagonal-norm NC operators only hold for the case of positive quadrature weights; 2) FD and FDB methods were implemented with their minimum number of stages; 3) the value for $q$ given for FD applies only to $q\ge 2$. The $^*$ denotes existing methods in the Runge-Kutta literature, and the $^\dagger$ denotes a method discussed in Section \ref{Sec:Methods}}
	\label{Tab:Abb}
\end{table}

%
\subsection{Efficiency Comparisons}
%

For the study of efficiency, the temporal domain is chosen to be $t\in[0,2]$ and the SAT penalty values are chosen such that both the temporal and spatial discretizations are dual-consistent. Two error measures are used. The first  is the stage error:
\begin{equation}
	e_\mathrm{stage} = \big|\big| \mathbf{e} \big|\big|_{B},
\end{equation}
where $B$ is a block diagonal matrix. For SBP and GSBP time-marching methods, the blocks are formed by the norm associated with method. For non-SBP Runge-Kutta methods, the diagonal of each block is populated with the entries of the $b$ coefficient matrix. The vector $\mathbf{e}$ contains the error in the numerical solution at the abscissa locations, integrated in space using the norm of the spatial discretization $H_s$:
\begin{equation}
	\mathbf{e}_{(j-1)n+k} = ||\mathbf{y}_{\mathrm{d},(j-1)n+k} -\mathcal{Y}((j+c_k)h)||_{H_s}, 
\end{equation}
where the subscripts $j=1,\ldots,N$ and $k=1,\ldots,n$ are the step and stage indices, respectively. By comparing with the exact solution of the IVP, we isolate the temporal error from the spatial error. The second error measure used is the solution error at the end of the final time step, integrated in space:
\begin{equation}
	e_\mathrm{step} = \big|\big| \tilde{y}_{\mathrm{d},N} - {\mathcal{Y}}(T) \big|\big|_{H_s}.
\end{equation}

Figure \ref{Fig:CD} shows the convergence of the stage and solution error with respect to CPU time in seconds for constant stage order, $\hat{q}=3$. The stage error, $e_\mathrm{stage}$, converges at the same rate for the various methods, as expected. Furthermore, the hierarchy in efficiency with respect to stage error negatively correlates with the number of stages in each method. Thus, the classical SBP time-marching methods, FD and FDB, are the least efficient due the their relatively large number of stages. Of those with four stages, the novel GSBP time-marching method LG is the most efficient method with respect to stage error. This scheme is more efficient than the well-known LGRI scheme, which has the same properties; however, it is not as efficient as the three-stage GRK or LGRII methods.

Considering the solution error, $e_\mathrm{step}$, the hierarchy of efficiency remains negatively correlated with the number of stages in each method of a given order $p$. The higher than expected convergence rate for the dense-norm LGL time-marching method ($p=\rho+1$) is only seen for linear problems.  As expected, the GSBP time-marching methods are more efficient than those based on classical SBP operators. This is especially true for those with a nonuniform abscissa or an abscissa which does not include $0$ or $1$. The most efficient method overall is GRK. It is one order lower than the GSBP time-marching methods LG and LGRI, but also has one less stage. This method is not L-stable. Eventually, the LG and LGRI methods become more efficient than the GRK method below an error of about $10^{-7}$.

\begin{figure}[!t]
	\centering
	{
	
	\begin{tabular}{cc}
	\hspace{-0.35cm} 
	\includegraphics[width=0.48\textwidth ]{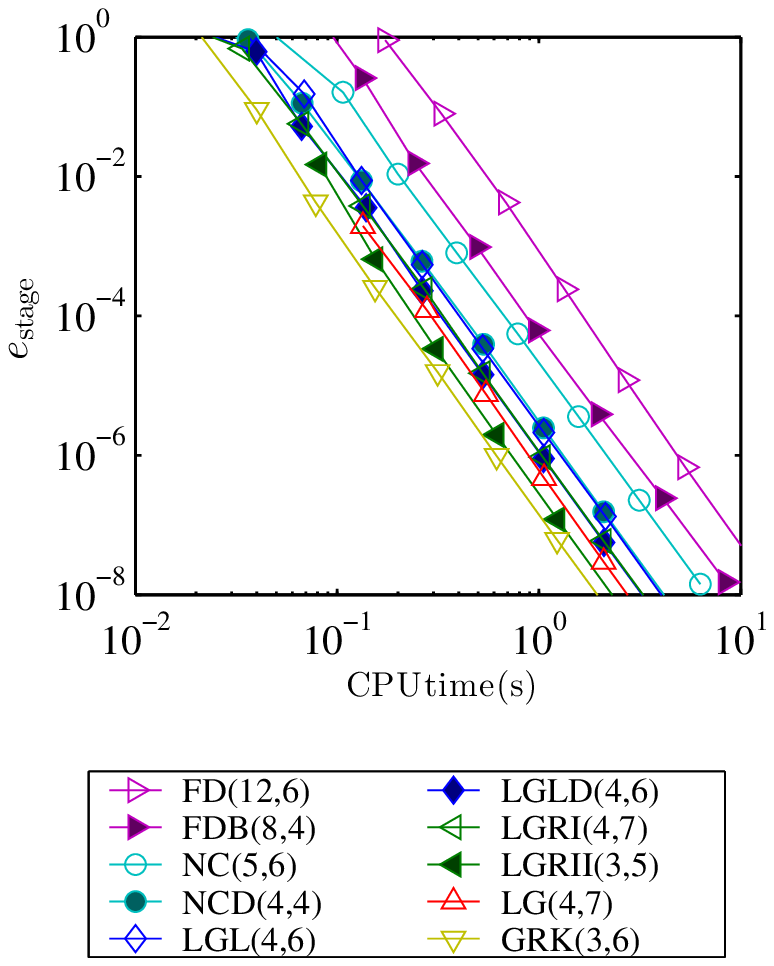} &
	\hspace{-0.55cm}
	\includegraphics[width=0.48\textwidth ]{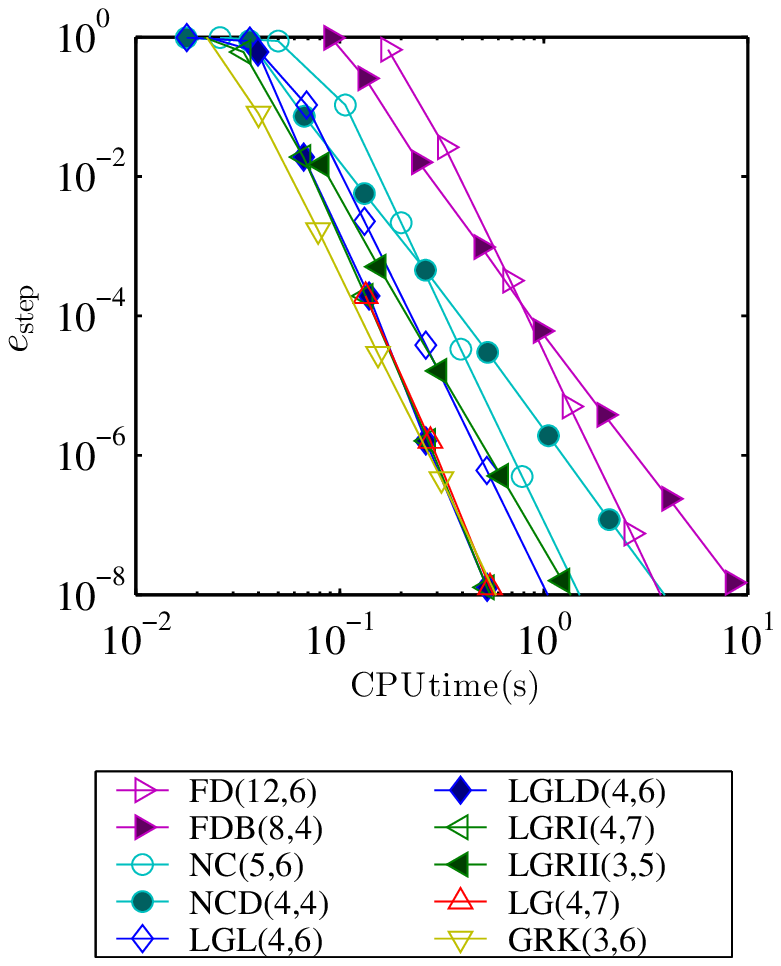}  \\
	\hspace{-0.35cm} 
	a) $e_\mathrm{stage}$ $(\hat{q}=3)$ &
	\hspace{-0.55cm}
	b) $e_\mathrm{step}$ $(\hat{q}=3)$ 
	\end{tabular}
	}
	\caption{{\bf Linear Convection Equation:} Convergence of the stage and solution error, $e_\mathrm{stage}$ and $e_\mathrm{step}$ respectively, with respect to CPU time $(s)$ for constant stage order. The numerical suffix in the legend indicates the number of stages in each time step $n$ and the order $p$. 	\label{Fig:CD}}
\end{figure}

Another perspective can be obtained by comparing methods of constant order $p$. Figure \ref{Fig:CD2} shows the convergence of the solution error with respect to step size $h = t_f-t_0$ and CPU time in seconds for constant order, $p=6$. This also includes the exclusively odd order GSBP time-marching methods based on Gauss quadrature of orders $p=5$ and $p=7$, as well as the fifth-order ESDIRK5 reference scheme. The error of classical SBP time-marching methods relative to time step size $h$ is significantly smaller than the GSBP time-marching methods. This however does not account for the higher number of stages. Therefore, the GSBP time-marching methods are nevertheless more efficient, as shown in Figure \ref{Fig:CD2} b). Apart from LGLD, which achieves higher than expected convergence, the well-known GRK scheme is the most efficient sixth-order scheme. As discussed above, the LG scheme of one order higher eventually becomes more efficient. It is also L-stable, which the GRK scheme is not.

\begin{figure}[!t]
	\centering
	{
	
	\begin{tabular}{cc}
	\hspace{-0.35cm} 
	\includegraphics[width=0.48\textwidth ]{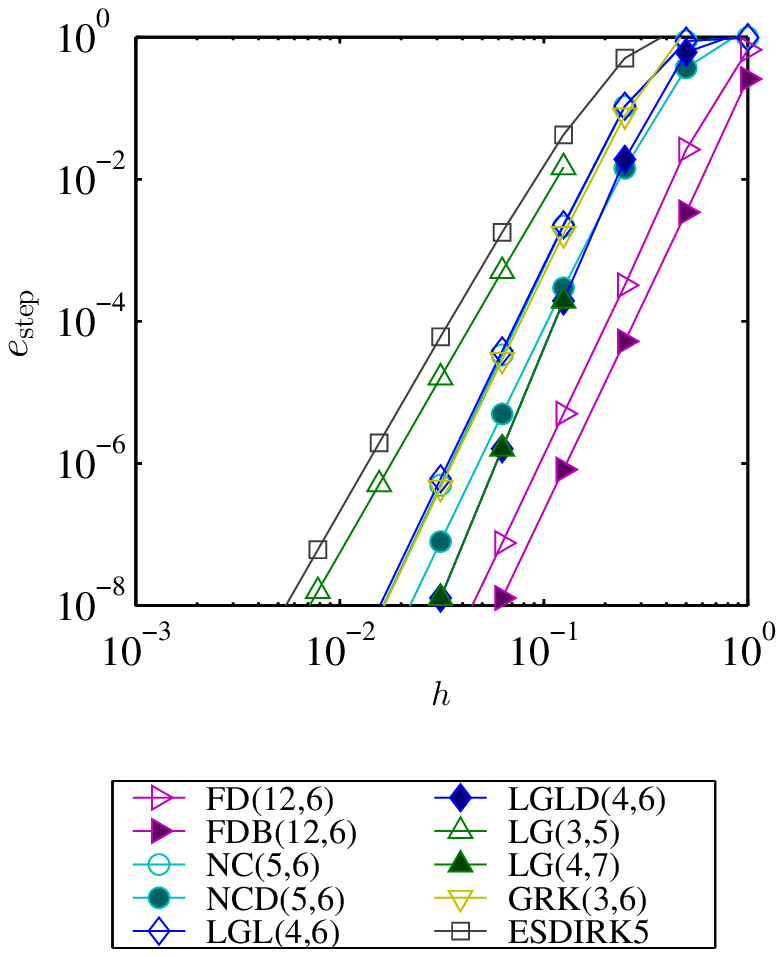} &
	\hspace{-0.55cm}
	\includegraphics[width=0.48\textwidth ]{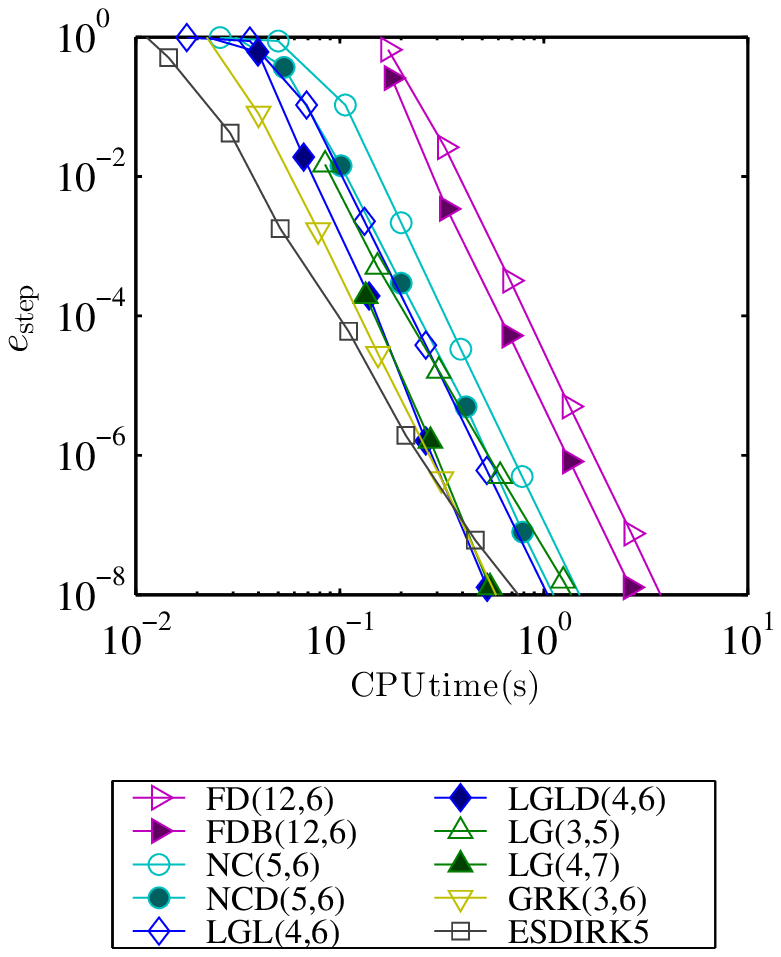}  \\
	\hspace{-0.35cm} 
	a) $e_\mathrm{stage}$ &
	\hspace{-0.55cm}
	b) $e_\mathrm{step}$   
	\end{tabular}
	}
	\caption{{\bf Linear Convection Equation:} Convergence of the solution error, $e_\mathrm{step}$, with respect to step size $h$ and CPU time $(s)$. The numerical suffix in the legend indicates the number of stages in each time step $n$ and the order $p$.}
	\label{Fig:CD2}
\end{figure}

%
\subsection{Diagonally-implicit methods}
%

While diagonally-implicit time-marching methods generally require a greater number of stages than fully-implicit schemes to achieve a prescribed order of accuracy, each stage can be solved sequentially. This replaces a single large system of equations for each time step with a smaller system of equations for each stage within a time step. The motivation for diagonally-implicit schemes comes from the nonlinear scaling of computational work with respect to the number of equations in a system.

For lower orders this effect is minimized as fully-implicit time-marching methods only require a few coupled stages. For example, the three-stage diagonally-implicit GSBP time-marching method developed in Section \ref{Sec:Methods} requires approximately the same computational effort as the two-stage LG, LGRI, and LGRII schemes of the same order. The fully-implicit schemes, however, have a much lower truncation error coefficient and are therefore more efficient.

As the order increases, so does the number of stages required by fully-implicit schemes. This is when diagonally-implicit scheme have the potential to be more efficient. As an example, consider the non-SBP ESDIRK5 scheme presented in Figure \ref{Fig:CD2}. It has the largest error as a function of step size $h$, but is the most efficient scheme considered above an error of about $10^{-7}$. This highlights the potential advantage of considering higher-order diagonally-implicit GSBP time-marching methods in the future.

%
\section{Conclusions}\label{Sec:Conc}
%

This article combines the generalized summation-by-parts framework originally presented in \cite{DBZ:2014} and the work of \cite{Nordstrom:2013,Nordstrom:2014} on the construction of time-marching methods based on FD-SBP operators. GSBP time-marching methods are shown to maintain the same stability and accuracy properties as those based on classical FD-SBP operators.  Specifically, all GSBP time-marching methods are shown to be L-stable: unconditional stability for linear IVPs along with damping of stiff modes. Those constructed with a diagonal norm are shown in addition to be BN-stable: unconditionally stable for contractive problems. The theory of superconvergent integral functionals, as well as the solution approximated at the end of a time step, is also extended to the generalized framework. The rate of superconvergence is shown to be connected to the accuracy with which the norm of the discretization approximates inner products of the primal and dual problem. This specifically includes the SAT term.

This article also shows the connection between SBP/GSBP time-marching methods and implicit Runge-Kutta methods. The connection to Runge-Kutta methods is used to derive minimum guarantee global order results for nonlinear problems. It is also used to derive the conditions under which BN-stable dense-norm GSBP time-marching methods can be constructed. While SBP/GSBP time-marching method form a subset of implicit Runge-Kutta methods, the SBP/GSBP characterization remains important. The approach simplifies the construction of high-order fully-implicit time-marching methods with a particular set of properties favourable for stiff IVPs. This can even lead to some novel Runge-Kutta schemes, as shown in the article. It also facilitates the analysis of fully-discrete approximations of PDEs and is amenable to multi-dimensional space-time discretizations. In the latter case, the explicit connection to Runge-Kutta methods is often lost.

A few examples of known and novel Runge-Kutta time-marching methods are presented which are associated with GSBP operators. This includes the known four-stage Lobatto IIIC method, a novel four-stage seventh-order Gauss-based fully-implicit scheme, and two novel diagonally-implicit schemes. The first is a three-stage third-order diagonally-implicit method, the second a fourth-order four-stage diagonally-implicit method. These methods are all L-stable and BN-stable. Numerical simulation of the linear convection equation is also presented to demonstrate the theory and to evaluate the relative efficiency of GSBP time-marching methods. In comparison with classical SBP time-marching methods, the GSBP based schemes considered are more efficient. Between GSBP time-marching methods, the novel Gauss based GSBP time-marching method retains the properties of the Radau IA scheme, and is slightly more efficient with respect to stage error. The global error however is comparable. Comparison with the non-SBP Gauss collocation methods is difficult as their orders do not match. For the same number of stages, the non-SBP method is one order higher and more efficient; however, it is not L-stable. When the SBP method is one order higher, the efficiency is comparable. Inclusion of a fifth-order ESDIRK scheme highlights the potential benefit of constructing higher-order GSBP time-marching methods in the future which are diagonally-implicit.

\section*{Acknowledgements}

The authors gratefully acknowledge the financial assistance of the Ontario Graduate Scholarship program and the University of Toronto.


\appendix

%
\section{The Runge-Kutta Abscissa of a GSBP Operator}\label{App:Abscissa}
%

This appendix examines the the application of SBP and GSBP operators and their associated projection operators to an abscissa  rescaled and translated: $\mathbf{c} = \frac{\mathbf{t} - \mathds{1}t_0^{[m]}}{h}$. These results simplify several proofs using the Runge-Kutta characterization in Section \ref{Sec:Runge-Kutta_Analogy}. This is presented as a series of three lemmas. The first defines the result of a GSBP first-derivative operator applied to the abscissa:

\begin{lemma}\label{Lem:AbDer}
	An SBP or GSBP first-derivative operator $D$ defined for the distribution of solution points $\mathbf{t}^{[m]}$ of order greater than or equal to $p$ applied to a monomial of the abscissa $\mathbf{c} = \frac{\mathbf{t} - \mathds{1}t_0^{[m]}}{h}$ of degree $p\ge 0$ yields
	\begin{equation}\label{Eq:Der_Ab}
		D \mathbf{c}^{p} = \frac{p}{h}\mathbf{c}^{p-1}.
	\end{equation}

\end{lemma}
\begin{proof}
	In this article exponentiation of vectors is computed element-wise, \textit{i.e.}
	\begin{equation}\label{Eq:ab2}
		\mathbf{c}^p = [c_1^p,\ldots,c_n^p]^T.
	\end{equation}
	Substituting the definition of the abscissa \eqref{Eq:GLM3} into \eqref{Eq:ab2} and expanding yields
	\begin{equation}\label{Eq:ab4}
		\mathbf{c}^p = \frac{1}{h^p} \sum_{i=0}^p \binom{p}{i}(\mathbf{t}^{[m]})^{p-i}(-t_0^{[m]})^i,
	\end{equation}
	where $\binom{n}{k} = \frac{n!}{k!(n-k)!}$ is the binomial coefficient. Applying an SBP or GSBP operator of order greater than or equal to $p$ yields
	\begin{equation}
		D \mathbf{c}^{p} = 
		\frac{1}{h^p} \sum_{i=0}^p \binom{p}{i}D(\mathbf{t}^{[m]})^{p-i}(-t_0^{[m]})^i =
		\frac{1}{h^p} \sum_{i=0}^{p-1} \binom{p}{i}(p-i)(\mathbf{t}^{[m]})^{p-1-i}(-t_0^{[m]})^i.
	\end{equation}
	Substituting the relationship $(p-k)\binom{p}{k}=p\binom{p-1}{k}$ gives
	\begin{equation}
		D \mathbf{c}^{p} = \frac{p}{h^p} \sum_{i=0}^{p-1} \binom{p-1}{i}(\mathbf{t}^{[m]})^{p-1-i}(-t_0^{[m]})^i
		 = \frac{p}{h}\mathbf{c}^{p-1}.
	\end{equation}
	The additional factor of $\frac{1}{h}$ comes from the norm matrix $H$ of the SBP or GSBP operator $D=H^{-1}\Theta$ defined for the interval $[t_0^{[m]},t^{[m]}_f]$ of size $h$, rather than $[0,1]$ on which the abscissa is defined. 
\end{proof}

Next, consider the application of the projection operator $\chi_{t_0}$ to the abscissa:

\begin{lemma}\label{Lem:Ab0}
	A projection operator $\chi_{t_0}$ defined for the distribution of points $\mathbf{t}^{[m]}$ of order greater than or equal to $p$ applied to a monomial of the abscissa $\mathbf{c} = \frac{\mathbf{t} - \mathds{1}t_0^{[m]}}{h}$ of degree $p\ge 0$ yields:
	\begin{equation}
		\chi_{t_0}\mathbf{c}^p = \bigg\lbrace
		\begin{array}{c}
			1,\ \mathrm{if}\ p=0\\
			0,\ \mathrm{if}\ p>0			
		\end{array}
		.
	\end{equation}
\end{lemma}
\begin{proof}
	Begin by expanding $\chi_{t_0}^T\mathbf{c}^p$ using \eqref{Eq:ab4}:
	\begin{equation}\label{Eq:ab3}
		\chi_{t_0}^T\mathbf{c}^p  = 
		\frac{1}{h^p} \sum_{i=0}^p \binom{p}{i}\chi_{t_0}^T(\mathbf{t}^{[m]})^{p-i}(-t_0^{[m]})^i=
		\frac{1}{h^p} \sum_{i=0}^p \binom{p}{i}(t_0^{[m]})^{p-i}(-t_0^{[m]})^i.
	\end{equation}
	Pulling out a factor of $(t_0^{[m]})^p$ from the summation and given that $\sum_{i=0}^p \binom{p}{i}(-1)^i=0$ for $p>0$, \eqref{Eq:ab3} simplifies to
	\begin{equation}\label{Eq:s0_Ab}
		\chi_{t_0}^T\mathbf{c}^p  = 0 \quad \text{for }p>0.
	\end{equation}
	If $p=0$ then $\chi_{t_0}^T\mathbf{c}^p  = \chi_{t_0}^T\mathds{1}  = 1$ by Definition \ref{Def:SBPGen}.
\end{proof}

Finally, consider the application of the projection operator $\chi_{t_f}$ to the abscissa:

\begin{lemma}\label{Lem:Abf}
	A projection operator $\chi_{t_f}$ defined for the distribution of points $\mathbf{t}^{[m]}$ of order greater than or equal to $p$ applied to a monomial of the abscissa $\mathbf{c} = \frac{\mathbf{t} - \mathds{1}t_0^{[m]}}{h}$ of degree $p\ge 0$ yields:
	\begin{equation}
		\chi_{t_0}\mathbf{c}^p = 1.
	\end{equation}
\end{lemma}
\begin{proof}
	Begin by expanding $\chi_{t_f}^T\mathbf{c}^p$ using \eqref{Eq:ab4}:
	\begin{equation}\label{Eq:ab5}
		\chi_{t_f}^T\mathbf{c}^p = 
		\frac{1}{h^p} \sum_{i=0}^p \binom{p}{i}\chi_{t_f}^T(\mathbf{t}^{[m]})^{p-i}(-t_0^{[m]})^i=
		\frac{1}{h^p} \sum_{i=0}^p \binom{p}{i}(t_f^{[m]})^{p-i}(-t_0^{[m]})^i.
	\end{equation}
	Likewise, $h^p=(t_f^{[m]}-t_0^{[m]})^p$ can be expanded as 
	\begin{equation}
		h^p=(t_f^{[m]}-t_0^{[m]})^p=\sum_{i=0}^p \binom{p}{i}(t_f^{[m]})^{p-i}(-t_0^{[m]})^i.
	\end{equation}
	Therefore, simplifying \eqref{Eq:ab5} gives
	\begin{equation}\label{Eq:s1_Ab}
		\chi_{t_f}^T\mathbf{c}^p  = 1 \quad \text{for }p\ge 0.
	\end{equation}

\end{proof}

These relationships greatly simplify the analysis, as the components of SBP or GSBP time-marching methods can be applied directly to the abscissae found in the Runge-Kutta conditions.

\bibliographystyle{siam}
\bibliography{M101491.bib}
\end{document}